\documentclass[leqno, 12pt]{amsart} 
\newtheorem{Theorem}{\indent\sc Theorem}[section]  
\newtheorem{Lemma}[Theorem]{\indent\sc Lemma}
\newtheorem{Proposition}[Theorem]{\indent\sc Proposition}
\newtheorem{Corollary}[Theorem]{\indent\sc Corollary}
\theoremstyle{definition}   
\newtheorem{Definition}[Theorem]{\indent\sc Definition} 
\theoremstyle{remark}
\newtheorem{Remark}[Theorem]{\indent\sc Remark}
\newtheorem{Example}[Theorem]{\indent\sc Example}

\numberwithin{equation}{section}
\usepackage{color, array, amssymb, amscd, bm}
\usepackage{graphicx}
\usepackage[normalem]{ulem}
\usepackage{amsmath}	
\newcommand{\R}{\mathbb R}
\newcommand{\C}{\mathbb C^{\prime}}

\newcommand{\D}{\mathbb D}

\newcommand{\Min}{\mathbb L^3}
\newcommand{\Nil}{{\rm Nil}_3}

\renewcommand{\Re}{\operatorname {Re}}
\renewcommand{\Im}{\operatorname {Im}}

\newcommand{\ip}{i^{\prime}}

\begin{document}
 \title[Timelike minimal surfaces with $Q \overline Q=0$ in $\Nil$]{Minimal null scrolls in the three-dimensional Heisenberg group}
%
 \author[H.~Kiyohara]{Hirotaka Kiyohara}
 \address{
 Department of Mathematics, Hokkaido University, 
 Sapporo, 060-0810, Japan}
 \email{kiyosannu@eis.hokudai.ac.jp}
 \thanks{This work was supported by JST SPRING, Grant Number JPMJSP2119.}
 \date{}

\begin{abstract}
 In the three-dimensional Heisenberg group equipped with a certain left invariant Lorentzian metric,
 timelike minimal surfaces
 which have the Abresch-Rosenberg differentials
 with vanishing multiplication of the coefficient function and its para-complex conjugate
 are characterized
 as the surfaces defined by the multiplication of null curves and affine null lines.
 Moreover, the constructions of these surfaces with prescribed curvatures of null curves or prescribed null lines are given.
\end{abstract}
\maketitle   

\section{Introduction}
The left invariant Lorentzian metrics
in the three-dimensional Heisenberg group $\Nil$
are isometrically classified into three types \cite{R1992}.
In this paper we investigate timelike minimal surfaces in $\Nil$ equipped with one of them,
which is non-flat and its isometry group is of dimension $4$.

In the classical surface theory in the $3$-dimensional space forms,
a quadratic differential called the {\it Hopf differential} is a useful tool
to study constant mean curvature surfaces.
Abresch and Rosenberg \cite{AR2004, AR2005} introduced another quadratic differential
for a surface in the homogeneous $3$-spaces which have isometry groups of dimension $4$.
It is holomorphic for constant mean curvature surfaces
and known as an analogy of the Hopf differential.

In the previous study \cite{KK2022},
Kobayashi and the author gave the corresponding quadratic differential $Qdz^2$ in our setting
and called it the {\it Abresch-Rosenberg differential}
for a timelike surface in the Lorentz Heisenberg group $\Nil$ (Definition \ref{defARdiff}).
It is a para-complex number valued $(0, 2)$-tensor field.
Moreover,
in \cite{KK2022},
we determined timelike minimal surfaces with the Abresch-Rosenberg differential vanishing.

We will determine timelike minimal surfaces with the condition $Q \overline Q=0$ in this paper.
We would like to note that
vanishing the real-valued function $Q \overline Q$ does not imply $Q=0$
because of using the para-complex coordinates.
The condition $Q \overline Q=0$ proceeds from the integrability condition \eqref{integrablecondition}
for a timelike minimal surface in $\Nil$.

As is well known,
null scrolls are ruled surfaces over null curves with null director curves in the Minkowski space.
Based on these surfaces,
we define a new class of surfaces in $\Nil$,
using the exponential map of the Lie group,
in the form$:$
$\gamma(s) \cdot \exp(t \widetilde B(s))$,
and call them {\it null scrolls} in $\Nil$
(Definition \ref{defnullscroll} and Remark \ref{remnullscroll}).
Here $\gamma$ is a null curve in $\Nil$
and $\widetilde B$ is
a curve in the light cone in Lie algebra $\mathfrak{nil}_3$.
The purpose of this paper is to show that
minimal null scrolls are the only timelike minimal surfaces
satisfying $Q \overline Q=0$,
where $Qdz^2$ is the Abresch-Rosenberg differential (Theorem \ref{thmmain}).
This theorem will be expected to give a geometrical interpretation
of the Abresch-Rosenberg differential.
In order to prove this theorem,
we introduce a null curve theory in $\Nil$ by using one in the Minkowski $3$-space.
In particular,
we prove that for any two real valued functions $k_1$ and $k_2$,
there exists a framed null curve with the``curvatures'' $k_1$ and $k_2$
(Theorem \ref{curvefromcurvatures}).
Moreover,
we obtain a method to construct minimal null scrolls
from an arbitrary real valued function (Corollary \ref{Corcurvature}).

In Section \ref{sec:pre},
we will recall the basic knowledge of the three-dimensional Heisenberg group,
the para-complex structure,
and integrability conditions of timelike minimal surfaces in $\Nil$.
In Section \ref{sec:characterization},
we will introduce null scrolls,
and prove the main theorem mentioned above.
Section \ref{sec:construct} presents a more practical construction of minimal null scrolls than the one in Section \ref{sec:characterization}.
We can construct all minimal null scrolls
from prescribed rulings $\widetilde B$ and an arbitrary function of one variable (Theorem \ref{thmmain2}).
In the Appendix,
we write down the background of this study.
In particular,
a relation between minimal null scrolls in $\Nil$
and $B$-scrolls in Minkowski $3$-space is explained.


\section{Preliminaries}\label{sec:pre}
We devote this section to recalling the fundamental theorem of timelike surfaces in the Minkowski $3$-space and to explaining the basic information of timelike surfaces in the three-dimensional Heisenberg group.
The Gauss-Codazzi equations for these surfaces reveal that
a timelike minimal surface in the three-dimensional Heisenberg group induces a timelike surface of constant mean curvature $1/2$ in the Minkowski $3$-space.


\subsection{Para-complex structure}
First, we give some basic properties of the para-complex number.
Para-complex number $\C$ is a real algebra spanned by $1$ and the para-complex unit $\ip$
which satisfies the conditions$:$
\begin{equation*}
 {\ip}^2=1 \quad \text{and} \quad 1 \cdot \ip = \ip \cdot 1 = \ip.
\end{equation*}
We call the elements of $\C$ also para-complex numbers.
Similarly to complex number $\mathbb{C}$,
every $z \in \C$ can be decomposed uniquely into the real part and the imaginary part,
that is,
\begin{equation*}
 z= x +\ip y.
\end{equation*}
Moreover, the para-complex conjugate is also defined as well as the complex conjugate$:$
\begin{equation*}
 \bar z = x - \ip y.
\end{equation*}
However, the para-complex number is not a field.
In fact, there exist numbers which do not have inverse elements.
This fact is explained in the following proposition.
\begin{Proposition}\label{properties}
 For a para-complex number $z = x + \ip y \in \C$,
 following statements hold.
 \begin{enumerate}
 \item
 There exists a root $w \in \C$,
 that is,
 $z=w^2$
 if and only if
 \begin{equation*}
  x+y \geq 0 \quad \text{and} \quad x-y \geq 0. 
 \end{equation*}
 \item
 There exists a para-complex number $w$ such that $z= e^w$ if and only if
 \begin{equation*}
  x+y>0 \quad \text{and} \quad x-y>0.
 \end{equation*}
 Here the exponential is defined as an infinite series
 \begin{equation*}
  e^w=\sum_{k=0}^\infty \frac{w^k} {k!}=e^{\Re w}\left( \cosh(\Im w)+ \ip \sinh(\Im w) \right).
 \end{equation*}
 \item
 There exists an inverse number $w$,
 that is,
 $z \cdot w=1$
 if and only if
 \begin{equation*}
 x^2-y^2\neq0.
 \end{equation*}
 \end{enumerate}
\end{Proposition}
\begin{proof}
 By representing $w^2$ and $e^w$ using the real part and the imaginary part of $w$,
 and comparing them with $z$,
 we can obtain the statements $(1)$ and $(2)$.
 If $z$ has a inverse number $w=u +\ip v$,
 we have a system
 \begin{equation*}
  xu+yv=1 \quad \text{and} \quad yu+xv=0.
 \end{equation*}
 Then removing $u$ or $v$, we obtain $x^2-y^2\neq0$.
 Conversely,
 if $x^2-y^2\neq0$ holds,
 multiplying $\epsilon \in \{\pm 1, \pm \ip\}$ appropriately results in case of $x+y>0$ and $x-y>0$.
 When $x+y>0$ and $x-y>0$,
 by $(2)$,
 $z$ can be represented into $e^{\widetilde w}$ for some para-complex number $\widetilde w$.
 The proof is completed by taking $w=e^{-\widetilde w}$.
\end{proof}

Let $M$ be a connected manifold
and $(\widetilde M, h)$ a Lorentzian manifold.
An immersion $f: M \to \widetilde M$ is said to be {\it timelike}
when the induced metric is Lorentzian.
If $M$ is $2$-dimensional,
a timelike immersion is said to be a {\it timelike surface}.
When $M$ is a orientable connected $2$-manifold,
the induced Lorentzian metric,
which we call the first fundamental form,
of a timelike surface $f: M \to \Nil$ determines a Lorentzian conformal structure on $M$
(for details, see \cite{W1996}).
Thus we can rewrite locally the first fundamental form $I$ as
\begin{equation*}
 I= e^u dz d\bar z = e^u (dx^2-dy^2)
\end{equation*}
using para-complex coordinates $z=x+\ip y$.
We call $e^u$ the {\it conformal factor}.
In this paper,
we use the notations $\partial$ and $\overline \partial$ as the partial differentiations with respect to the coordinates $z$ and $\bar z$,
defined formally by
\begin{equation*}
\partial = \frac{\partial}{\partial z} = \frac12 \left( \frac{\partial}{\partial x} + \ip \frac{\partial}{\partial y} \right)
\quad
\text{and}
\quad
\overline \partial = \frac{\partial}{\partial \bar z} = \frac12 \left( \frac{\partial}{\partial x} - \ip \frac{\partial}{\partial y} \right).
\end{equation*}
Furthermore,
for a timelike surface $f$,
we denote tangent vectors $df( \partial )$ and $df( \overline \partial )$
by $f_z$ and $f_{\bar z}$, respectively.


\subsection{Timelike surface theory in the $3$-dimensional Minkowski space}
The Minkowski $3$-space is a real $3$-space $\mathbb{R}^3$ equipped with the standard Lorentzian metric$:$
\[
 \langle (x_1, x_2, x_3) , (\tilde x_1, \tilde x_2, \tilde x_3) \rangle := -x_1 \tilde x_1 + x_2 \tilde x_2 + x_3 \tilde x_3.
\]
We will use the notation $\Min_{(-,+,+)}$ as the Minkowski $3$-space with the above metric in this paper.
Moreover we denote the Lorentz group by $O(2,1)$.
This group is defined as follows$:$
\[
 O(2,1)
 =
 \left\{
  X \in M_3\mathbb{R}
 \middle|
  X^\top \begin{pmatrix}-1&0&0\\0&1&0\\0&0&1\end{pmatrix} X= \begin{pmatrix}-1&0&0\\0&1&0\\0&0&1\end{pmatrix}
 \right\}.
\]
The following transformation for $X \in O(2,1)$ is called a Lorentz transformation.
\[
 \Min_{(-,+,+)} \ni (x_1, x_2, x_3) \mapsto (x_1, x_2, x_3) X \in \Min_{(-,+,+)}.
\]
The isometry group of $\Min_{(-,+,+)}$ is the Poincar\'{e} group which consists of the compositions of Lorentz transformations and parallel translations.
In particular,
orientation preserving isometries are made from translations and the special Lorentz group $SO(2,1)$$:$
\[
 SO(2,1)
 =
 \left\{
  X \in O(2,1)
 \middle|
  \det X=1
 \right\}.
\]

Let $\D \subset \C$ be a simply connected domain of $\C$
and $\Phi: \D \to \Min_{(-,+,+)}$ be a conformal immersion with the conformal factor $e^{\widetilde{u}}$.
The unit normal vector field $\widetilde{N}$ is defined by
\[
 \widetilde{N}= -\ip \frac{\Phi_z \times \Phi_{\bar z}} {|\Phi_z \times \Phi_{\bar z}|},
\]
and the second fundamental form $II$ is given by
\[
 II
 =
 \langle \Phi_{zz}, \widetilde{N} \rangle dz^2 + 2 \langle \Phi_{z\bar{z}}, \widetilde{N} \rangle dzd\bar{z} + \langle \Phi_{\bar{z}\bar{z}}, \widetilde{N} \rangle d\bar{z}^2.
\]
The mean curvature $\widetilde{H}$ and the Hopf differential $\widetilde{Q}dz^2$ of $\Phi$ are the trace of $II$ with respect to the induced metric on $\D$ and the $(2,0)$-part of $II$.
As is well known,
the Gauss-Codazzi equations for timelike surfaces are given as
\begin{equation}\label{eq:gausscodazziminkowski}
   \frac12 \widetilde{u}_{z \bar z} + \frac14 {\widetilde{H}}^2 e^{\widetilde{u}} - \widetilde{Q} \overline{\widetilde{Q}}e^{-\widetilde{u}}=0,
   \quad
   \widetilde{Q}_{\bar z}=\frac12 \widetilde{H}_z e^{\widetilde{u}}.
\end{equation}
\begin{Theorem}\label{thm:fundamentaltheoremminkowski}
 Let triplet $(\widetilde{u}, \widetilde{H}, \widetilde{Q})$ be a solution of the equation \eqref{eq:gausscodazziminkowski}.
 Then there uniquely exists a conformal immersion $\Phi: \D \to \Min_{(-,+,+)}$ which has the conformal factor $e^{\widetilde{u}}$, the mean curvature $\widetilde{H}$, and the Hopf differential $\widetilde{Q}dz^2$ up to orientation preserving isometries.
\end{Theorem}


\subsection{$B$-scrolls}
Let us introduce surfaces in $\Min_{(-,+,+)}$, called {\it B-scrolls},
first introduced by \mbox{L. Graves} \cite{G1979}.
$B$-scrolls are timelike ruled surfaces, ruled by the binormal vector fields of null Frenet frames along base null curves.
Precisely, defined as the following definition$:$
\begin{Definition}
 Let $\gamma=(\gamma_1, \gamma_2, \gamma_3)$ be a null curve in $\Min_{(-,+,+)}$ defined on an open interval.
 Then $\gamma$ with an appropriate parameter $s$ has a frame $(A, B, C)$ which satisfies
 \begin{align*}
  A&= \gamma^{-1} \frac{d\gamma}{ds},\quad 
  \langle A, B \rangle = \langle C, C \rangle =1,\notag\\
  \langle A, A \rangle &= \langle B, B \rangle = \langle A, C \rangle = \langle B, C \rangle =0,\label{nullframe}\\
  (A', B&', C')
  =
  (A, B, C)\begin{pmatrix} 0&0&-k_2\\0&0&-k_1\\k_1&k_2&0\end{pmatrix}\notag.
 \end{align*}
 The frame $(A, B, C)$ is called a {\it null Frenet frame} of $\gamma$,
 and $k_1, k_2$ the {\it first curvature} and the {\it second curvature} of $\gamma$ with respect to $(A, B, C)$.
 The ruled surface defined as the following formula is called a $B$-scroll of $\gamma$$:$
 \begin{equation*}
  \Phi(s, t)=\gamma(s) + t B(s).
 \end{equation*}
 The null curve $\gamma$ is called a {\it base curve} of a $B$-scroll $\Phi$.
\end{Definition}
\begin{Remark}\label{rem:B-scrolls}
$(1)$ \indent
The structure of $B$-scroll is preserved under the isometries
because a null frame is preserved under the isometries of $\Min_{(-,+,+)}$.
Precisely,
for a null curve $\gamma$ with a null Frenet frame $(A, B, C)$ and curvatures $k_1, k_2$ and an isometry $\Psi: \bm{x} \mapsto \bm{x} F_0 + \bm{b}$ of $\Min_{(-,+,+)}$,
the transformed curve $\Psi \circ \gamma$ is null and has a null Frenet frame $F_0 (A, B, C)$ and curvatures $k_1, k_2$.\\
$(2)$ \indent
Since a null Frenet frame of a null curve depends on the choice of a screen bundle,
$B$-scrolls are not unique for a base null curve.
For details about screen bundles, see \cite{DJ2007} or subsection \ref{sec:nullframe}.
The mean curvature of a $B$-scroll is given by the second curvature $k_2$ of the base curve $\gamma$ with respect to a fixed screen bundle.\\
$(3)$ \indent
It is known that the Hopf differential $\widetilde Qdz^2$ of $B$-scrolls satisfies $Q \overline{\widetilde{Q}} =0$.
Moreover, timelike surfaces with $Q \overline{\widetilde{Q}}=0$ and $Q \neq 0$ are only $B$-scrolls \cite{C2012, FI2003, SGP2019}.
\end{Remark}


\subsection{$3$-dimensional Heisenberg group $\Nil$}
The $3$-dimensional Heisenberg group $\Nil$ is a linear Lie group
consisting of $3\times 3$ upper-triangular matrices
which have the diagonal components $1$,
\begin{equation*}
 \Nil = 
 \left\{
   \begin{pmatrix}
     1 & x_1 & x_3 + \frac12 x_1 x_2\\
     0 & 1 & x_2\\
     0 & 0 & 1
   \end{pmatrix}
 \middle|
   x_1, x_2, x_3 \in \R
 \right\}.
\end{equation*}
It can be considered as $(\R^3, \cdot)$ where $\cdot$ denotes the group structure defined by
\begin{equation*}
 (x_1, x_2, x_3) \cdot (\tilde {x}_1, \tilde {x}_2, \tilde {x}_3)
 =
 \left( x_1+ \tilde x_1, x_2 + \tilde x_2, x_3 + \tilde x_3 + \frac12 (x_1 \tilde x_2 - x_2 \tilde x_1) \right).
\end{equation*}
Let $\mathfrak{nil}_3$ denote the Lie algebra for $\Nil$.
Then $\mathfrak{nil}_3$ is $\R^3$ endowed with the Lie bracket $[\ ,\  ]$ which satisfies
\begin{equation*}
 [e_1, e_2] = e_3, \quad [e_2, e_3] = [e_3, e_1] = 0,
\end{equation*}
for the standard basis $\{e_1, e_2, e_3 \}$ of $\R^3$.
The exponential map $\exp: \mathfrak{nil}_3 \to \Nil$ is a diffeomorphism expressed by
\begin{equation*}
 \exp (x_1e_1 + x_2e_2 + x_3e_3) = (x_1, x_2, x_3).
\end{equation*}
As is well known,
$3$-dimensional Heisenberg group has the left invariant Riemannian metric
\begin{equation*}
 g_R= d{x_1} \otimes d{x_1} + d{x_2} \otimes d{x_2} + \omega \otimes \omega
 \quad \text{where} \quad
 \omega= dx_3 + \frac12 (x_2 dx_1 -x_1 dx_2).
\end{equation*}
On the other hand, three types of left invariant Lorentzian metrics on $\Nil$ exist in terms of isometrical classifications \cite{R1992}.
In particular,
non-flat left invariant Lorentzian metrics are given as
\begin{equation*}
 g_{\pm} = \mp d{x_1} \otimes d{x_1} + d{x_2} \otimes d{x_2} \pm \omega \otimes \omega.
\end{equation*}
Then an orthonormal basis $\{ E_1, E_2, E_3 \}$
of each tangent space of $\Nil$
with respect to $g_R$, $g_{+}$ or $g_{-}$
can be obtained
with the left invariant vector fields which correspond to $e_1$, $e_2$, and $e_3$$:$
\begin{equation*}
 E_1 = \frac{\partial}{\partial x_1} - \frac{x_2}2 \frac{\partial}{\partial x_3},
 \quad
 E_2 = \frac{\partial}{\partial x_2} + \frac{x_1}2 \frac{\partial}{\partial x_3},
 \quad \text{and} \quad
 E_3 = \frac{\partial}{\partial x_3}.
\end{equation*}
\begin{Remark}
 The $1$-form $\omega$ is a contact form on $\Nil$
 and then $(\Nil, \omega)$ becomes a contact manifold.
 For the contact manifold $(\Nil, \omega)$,
 the Reeb vector field is given by $E_3$.
\end{Remark}
The volume element with respect to the metric $g_R$, $g_{+}$ or $g_{-}$
is given by $\pm dx_1 \wedge dx_2 \wedge dx_3$.
We then orientate $\Nil$ for the volume form to be $dx_1 \wedge dx_2 \wedge dx_3$.
Moreover, we define the vector product $\times$ as
\begin{equation*}
 g(X \times Y, Z) = dx_1 \wedge dx_2 \wedge dx_3 (X, Y, Z),
 \quad
 \text{for}
 \quad
 X, Y, Z \in \Gamma(T\Nil)
\end{equation*}
where $g=g_R$, $g_{+}$ or $g_{-}$.

In this paper,
we assign the left invariant Lorentzian metric
\begin{equation*}
 g_{+} = - d{x_1} \otimes d{x_1} + d{x_2} \otimes d{x_2} + \omega \otimes \omega
\end{equation*}
on $\Nil$,
and denote $g_+$ as simply $g$.
Note that the vector field $E_1$ supplies a timelike vector
and the vector fields $E_2$ and $E_3$ supply spacelike vectors at each point.

Let $\nabla$ denote the Levi-Civita connection for the left invariant Lorentzian metric $g$,
and then the connection $\nabla$ is defined as
\begin{equation*}\begin{matrix}
 \nabla _{E_1} E_1 = 0,                   & \nabla _{E_1} E_2 = \frac12 E_3, & \nabla _{E_1} E_3 = -\frac12 E_2,\vspace{2pt}\\
 \nabla _{E_2} E_1 = -\frac12 E_3, & \nabla _{E_2} E_2 = 0,                 & \nabla _{E_2} E_3 = -\frac12 E_1,\vspace{2pt}\\
  \nabla _{E_3} E_1 = -\frac12 E_2, &\nabla _{E_3} E_2 = -\frac12 E_1,   & \nabla _{E_3} E_3 = 0. 
\end{matrix}\end{equation*}


\subsection{Integrability conditions of timelike minimal surfaces}
Let $f: \D \to \Nil$ be a conformal immersion
from a simply connected domain $\D \subset \C$ into the $3$-dimensional Heisenberg group $\Nil$,
$N$ be the unit normal vector field,
$z$ be a conformal coordinate on $\D$,
and $e^u$ be the conformal factor of $f$.
Then the {\it support function} $h$ of $f$ is defined by
\[ h=-e^{u/2} g(N, E_3). \]
When the support function $h$ vanishes
the vector field $E_3$ is tangent to the surface.
Therefore timelike surfaces with $h$ vanishing identically are part of the inverse image
of plane curves by the natural projection $\pi: \Nil \to \R^2$ that extracts the first and second components
\begin{equation*}
 \pi(x_1, x_2, x_3)=(x_1, x_2).
\end{equation*}
These surfaces are called {\it Hopf cylinders}.
We call a timelike minimal surface {\it non-vertical} if the support function $h$ does not vanish anywhere.

In \cite{AR2005} Abresch and Rosenberg introduce a quadratic differential,
called the {\it Abresch-Rosenberg differential},
into surfaces in Thurston's model manifolds.
It is known as an analogy of the Hopf differential of surfaces in space forms.
Even in the case of timelike surfaces in $\Nil$,
the Abresch-Rosenberg differential can be defined in a similar way as in the positive definite case.
We will denote the components of $f_z$ by $\phi^j$ $(j=1, 2, 3)$,
that is,
\[ f_z = \sum_{j=1}^3 \phi^j E_j.\]
\begin{Definition}\label{defARdiff}
 The Hopf differential for $f$ is given by the $(2,0)$-part of the second fundamental form,
 that is,
 \begin{equation*}
  \widetilde Q = g(\nabla_{\partial} f_z, N),
 \end{equation*}
 and define a para-complex valued function $Q$ by
 \begin{equation*}
  Q= \frac{2H-\ip}4 \widetilde Q - \frac{(\phi^3)^2}4,
 \end{equation*}
 where the function $H$ denotes the mean curvature of $f$.
 The quadratic differential $Qdz^2$ is well-defined,
 and then it is called the Abresch-Rosenberg differential for a timelike surface $f$.
\end{Definition}
The Integrability conditions of non-vertical timelike minimal surfaces are denoted
by using the support function and the Abresch-Rosenberg differential$:$
\begin{equation}\label{integrablecondition}
 \frac12 w_{z \bar z} + e^w -Q \overline Q e^{-w}=0
 \quad \text{and} \quad
 Q_{\bar z}=0.
\end{equation}
Here,
the real valued function $w$ is defined by $|h|=4e^{w/2}$.
For a solution $(w, Q)$ of \eqref{integrablecondition},
there exists a timelike minimal surface
which has the support function $h$ with $|h|=4 e^{w/2}$ and the Abresch-Rosenberg differential $Qdz^2$.
It is known that there exists a relation between timelike minimal surfaces in $\Nil$
and timelike constant mean curvature $1/2$ surfaces in $\Min_{(-,+,+)}$.
Indeed, the integrability conditions \eqref{integrablecondition} of a timelike minimal surface in $\Nil$ coincide with
the Gauss-Codazzi equations \eqref{eq:gausscodazziminkowski} for a timelike constant mean curvature $1/2$ surface in $\Min_{(-,+,+)}$
when the product of the Abresch-Rosenberg differential by $4$
and the square of the support function
denote the Hopf differential and the conformal factor of a surface in $\Min_{(-,+,+)}$$:$
\[
   4Q=\widetilde Q, \quad h^2=e^{\widetilde{u}}.
\]
Therefore a timelike minimal surface in $\Nil$
with the Abresch-Rosenberg differential $Qdz^2$
and the support function $h$ induces a timelike surface in $\Min_{(-,+,+)}$
with the constant mean curvature $1/2$,
the first fundamental form $h^2 dz d\bar z$,
and the Hopf differential $4Qdz^2$.
These surfaces have the following relations$:$
If we take a timelike minimal surface $f$ in $\Nil$ with the derivative $f_z=\phi^1E_1 + \phi^2E_2 + \phi^3E_3$,
then the induced surface $\Phi$ in $\Min_{(-,+,+)}$ has the derivative $\Phi_z=(\phi^1, \phi^2, \ip \phi^3)$.
Moreover,
the normal Gauss map of $f$ coincides with
the Gauss map of $\Phi$.
Here the normal Gauss map is a composition of the left-translated unit normal vector field, stereographic projection of the de-Sitter sphere in $\mathfrak{nil}_3$, the operation to switch the component of the real part and imaginary part of para-complex numbers, and the inverse map of the stereographic projection of the de-Sitter sphere in $\Min_{(-,+,+)}$.
I would like to want you to see \cite{KK2022} for details of the integrability conditions of timelike surfaces in $\Nil$ and the above relations.
Then these relations derive the following theorem.
\begin{Theorem}\label{thmnonuniqueness}
 Let $f_1$ and $f_2$ be timelike minimal surfaces in $\Nil$,
 $h_j$ be the support functions,
 and $Q_jdz^2$ be the Abresch-Rosenberg differentials of $f_j$ for $j=1, 2$.
 If $h_1= \pm h_2$ and $Q_1= Q_2$ hold,
 then the derivatives ${f_1}_z=\sum\phi_1^kE_k$ and ${f_2}_z= \sum \phi_2^kE_k$ satisfy
 \begin{equation}\label{nonuniqueness}
  ( \phi_2^1, \phi_2^2, \ip \phi_2^3)
   =
   (\phi_1^1, \phi_1^2, \ip \phi_1^3) F_0
 \end{equation}
 for some matrix $F_0 \in SO(2,1)$.
 Conversely,
 if $f_1$ and $f_2$ satisfy \eqref{nonuniqueness},
 then $h_1=\pm h_2$ and $Q_1= Q_2$ hold.
\end{Theorem}
\begin{proof}
 Let $f_1$ and $f_2$ be timelike minimal surfaces in $\Nil$
 with the same or multiplied by $-1$ support function $h$ and the same Abresch-Rosenberg differential $Qdz^2$.
 Then the surfaces $\Phi_1$ and $\Phi_2$ induced from $f_1$ and $f_2$
 have the same metric $h^2dzd\bar z$ and the same Hopf differential $4Qdz^2$.
 Therefore Theorem \ref{thm:fundamentaltheoremminkowski} shows that
 there exists an orientation preserving isometry $\Psi$ of $\Min_{(-,+,+)}$ such that $\Phi_2= \Psi \circ \Phi_1$.
 This implies that \eqref{nonuniqueness} holds for some $F_0 \in SO(2,1)$.

 Conversely we assume that timelike minimal surfaces $f_1$ with $h_1, Q_1$ and $f_2$ with $h_2, Q_2$ satisfy \eqref{nonuniqueness} for some $F_0 \in SO(2,1)$.
 Then the induced surfaces $\Phi_1$ and $\Phi_2$ in $\Min$ are same up to orientation preserving isometries of $\Min_{(-,+,+)}$.
 Therefore $\Phi_1$ and $\Phi_2$ have the same first fundamental form and the same Hopf differential by Theorem \ref{thm:fundamentaltheoremminkowski}.
 The relation between timelike minimal surfaces in $\Nil$ and timelike constant mean curvature $1/2$ surfaces in $\Min_{(-,+,+)}$
 shows $ h_2 = \pm h_1$ and $Q_2= Q_1$.
\end{proof}
\begin{Remark}
 A Weierstrass type representation of timelike minimal surfaces in $\Nil$ is studied in \cite{KK2022}.
 The Weierstrass data of this representation is determined by the Abresch-Rosenberg differential and the support function.
 We can construct timelike minimal surfaces by integrating the Weierstrass data,
 and applying the integral to the loop group decomposition theorem and the Sym-Bobenko type representation.
 In general,
 timelike minimal surfaces corresponding to different initial conditions of the integrals of a Weierstrass data
 have the same or multiplied by $-1$ support function and the same Abresch-Rosenberg differential
 but they are not isometric (see \cite[Remark $5.5$]{KK2022}).
\end{Remark}

When the Abresch-Rosenberg differential satisfies $Q \overline Q=0$,
the first equation of the conditions \eqref{integrablecondition} becomes a hyperbolic type Liouville equation.
The exact solutions of Liouville equations are well-studied in \cite{BMV1994, CK1989, C1997, L1853, P1993}.
%


\section{Characterization of timelike minimal surfaces with $Q \overline Q=0$}\label{sec:characterization}
In this section,
we will characterize the timelike minimal surfaces with $Q \overline Q=0$.
We will call these surfaces {\it minimal null scrolls}
since they are considered as a generalization of ruled surfaces over null curves with null director curves in the Minkowski space,
called null scrolls.
For a null curve $\gamma$ in $\Nil$,
we will call the null vector field
$\gamma^{-1} \frac{d\gamma}{ds} =A=\sum_{i=1}^3 A^ie_i$
the {\it velocity} of $\gamma$.
Here the notation $\gamma^{-1}$ denotes the left translation from $T_{\gamma(\cdot)} \Nil$ to $\mathfrak{nil}_3$.


\subsection{Minimal surfaces in $\Nil$ which induce $B$-scrolls in $\Min_{(-,+,+)}$}\label{subsec:Bscrollnullscroll}
First, we will derive the timelike minimal surfaces which induces $B$-scrolls of constant mean curvature $1/2$ in this subsection.
From now on set $k_2= 1/2$.
Moreover for simplicity, we assume that curves in $\Min_{(-,+,+)}$ pass the origin at time $0$, that is, a curve $\gamma(s)$ satisfies $\gamma(0)=(0, 0, 0)$.

For any $B$-scroll $\Phi(s, t) = \gamma(s) + tB(s)$, a conformal coordinate $z= l x+\bar l y$ can be given by
\begin{equation*}
 s=x,
 \quad
 t=\frac{1}{x/8 + 1/y}.
\end{equation*}
Then a direct computation gives us the derivative $\Phi_z=(\phi_1, \phi_2, \ip\phi_3)$ as
\begin{equation*}
 \phi_1 = l (A^1(s) +t{B^1}'(s) - \frac{t^2}8 B^1(s)) +\bar l \frac{t^2}{y^2} B^1(s),
\end{equation*}
\begin{equation*}
 \phi_2 = l (A^2(s) +t{B^2}'(s) - \frac{t^2}8 B^2(s)) +\bar l \frac{t^2}{y^2} B^2(s),
\end{equation*}
\begin{equation*}
 \phi_3 = l (A^3(s) +t{B^3}'(s) - \frac{t^2}8 B^3(s)) -\bar l \frac{t^2}{y^2} B^3(s).
\end{equation*}
Here $A=(A^1, A^2, A^3)$ and $B=(B^1, B^2, B^3)$.
\begin{Lemma}\label{lemclosedform}
 \begin{equation}\label{closedform}
 2\Re\left( (\phi_1, \phi_2, \phi_3- \frac{\phi_1}2 \int_0^z \Re(\phi_2 dz) + \frac{\phi_2}2 \int_0^z \Re(\phi_1 dz)) dz \right)
 \end{equation}
 is a closed form.
\end{Lemma}
\begin{proof}
 Exterior derivative of \eqref{closedform} can be computed as
 \begin{equation*}
  (0, 0, 2 \partial \overline{\phi_3} + \phi_1 \overline{\phi_2} - \phi_2 \overline{\phi_1}) dz \wedge d\bar z.
 \end{equation*}
 We can check that this $2$-form vanishes by using the structure equations of the surface.
\end{proof}
By Lemma \ref{lemclosedform} and the Stokes' theorem,
the $\Nil$-valued map
\begin{equation*}
f(z,\bar z)
=(f_1, f_2, f_3)
=\int_0^z \Re\left( (\phi_1, \phi_2, \phi_3- \frac{\phi_1}2 \int_0^z \Re(\phi_2 dz) + \frac{\phi_2}2 \int_0^z \Re(\phi_1 dz)) dz \right)
\end{equation*}
is well-defined.
Moreover we can easily see that 
\begin{equation*}
 f_1= \int_0^z \Re(\phi_1 dz) =\gamma_1(s) + t B^1(s)
 \quad \text{and} \quad
 f_2= \int_0^z \Re(\phi_2 dz) =\gamma_2(s) + t B^2(s).
\end{equation*}
Let us represent $f_3=\int_0^z \Re\left((\phi_3- \frac{\phi_1}2 f_2 + \frac{\phi_2}2 f_1) dz\right)$ in terms of the coordinate $(s, t)$.
By the well-definedness of $f_3$,
we can choose an integral path from $0$ to $z= ls + \bar l /(-s/8 + 1/t)$ as
\begin{equation*}
 \Gamma_s: z(\tilde s)= l \tilde s, \quad 0 \leq \tilde s \leq s,
\end{equation*}
\begin{equation*}
 \Gamma_t: z(\tilde t)= ls + \bar l \frac{1}{-s/8 + 1/{\tilde t}}, \quad 0 \leq \tilde t \leq t.
\end{equation*}
A straightforward computation shows $f_3$ can be represented as
\begin{equation*}
 f_3(z, \bar z)
 =
 \gamma_3(s) +\int_0^s \left(- \frac{\gamma_2}2 A^1 + \frac{\gamma_1}2 A^2\right) ds
 +
 t\left(-B^3(s) -\frac{\gamma_2(s)}2 B^1(s) + \frac{\gamma_1(s)}2 B^2(s)\right).
\end{equation*}
Let us consider the maps
$\widetilde \gamma=(\gamma_1, \gamma_2, \gamma_3 + \int_0^s \left(- \frac{\gamma_2}2 A^1 + \frac{\gamma_1}2 A^2\right) ds)$
and $f=(f_1, f_2, f_3)$ as a curve in $\Nil$ and a map into $\Nil$, respectively.
Then the curve $\widetilde \gamma$ becomes a null curve since the velocity is given by
${\widetilde \gamma}^{-1} d\widetilde \gamma/ds = A^1e_1 + A^2e_2 + A^3e_3$,
and $f$ can be represented as
\begin{equation}\label{eq:B-scrolltypeminimal}
 f(z, \bar z)= \widetilde \gamma(s) \cdot \exp(t(B^1e_1 +B^2e_2- B^3e_3)).
\end{equation}


\subsection{Null scrolls and their minimality conditions}
Since the timelike surfaces which induce $B$-scrolls are given in the form of \eqref{eq:B-scrolltypeminimal}, the surfaces defined hereafter, named null scrolls, are worth considering.
In this subsection, we will derive their minimality conditions through some basic calculations.
For an arbitrary $\mathfrak{nil}_3$-valued vector field $X(s)=\sum_{j=1}^3 X^j(s) e_j$ defined on an open interval,
we denote the derivative (not covariant derivative) of $X$ by $X'$,
that is,
$X'(s)= \sum_{j=1}^3 {X^j}'(s) e_j$.
Here the \mbox{notation $'$} denotes the derivation with respect to the parameter $s$.

\begin{Definition}\label{defnullscroll}
 A timelike surface $f$ into $\Nil$ is said to be a {\it null scroll}
 if
 there exist a null curve $\gamma=\gamma(s)$ in $\Nil$
 and a curve $\widetilde B=\widetilde B(s)$ which takes values in the light cone in $\mathfrak{nil}_3$
 such that $f$ can be represented as
 \begin{equation}\label{scroll}
  f(s, t)= \gamma(s) \cdot \exp( t\widetilde B(s) ).
 \end{equation}
 Here,
 $\exp: \mathfrak{nil}_3 \to \Nil$ is the exponential map of Lie group $\Nil$.
 Moreover,
we call the curve $\gamma$ a {\it base curve}
and $\widetilde B$ a {\it ruling} of a null scroll $f$.
\end{Definition}
\begin{Remark}\label{remnullscroll}
 The maps of the equation \eqref{scroll} can also be defined for arbitrary Lie groups.
 For example,
 in semi-Euclidean spaces,
 the exponential map is the identity map and the group structure is given by the usual sum.
 Then the maps defined as \eqref{scroll} are ruled surfaces.
 Moreover,
 in Lie groups endowed with bi-invariant metrics,
 the exponential maps define geodesics,
 and then the maps of the form \eqref{scroll} become ruled surfaces.
 Therefore maps of \eqref{scroll} can be considered as a natural generalization of ruled surfaces.
\end{Remark}
Let $f(s,t) = \gamma(s) \cdot \exp( t\widetilde B(s) )$ be a null scroll
and $A$ be the velocity of $\gamma$.
Moreover expand $\widetilde B= \sum_{j=1}^3 B^je_j$.
Since the derivations of $f$ with respect to $s$ and $t$ are computed as
\begin{align*}
 f_s
 =&
 (A^1+t{B^1}') E_1 + (A^2 + t{B^2}') E_2\\
  &+ \left( A^3 +t({B^3}' + A^1B^2-A^2B^1) + \frac{t^2}2 (B^2{B^1}'- B^1{B^2}') \right) E_3,\\
 f_t
 =&
 B^1E_1 + B^2E_2 + B^3E_3,
\end{align*}
the first fundamental form $I$ of $f$ is given by
\begin{equation*}
 I=
 g_{11}ds^2 +2 g_{12}dsdt
\end{equation*}
where $g_{ij}$ are computed as
\begin{align*}
 g_{11}=&g(f_s, f_s)\\
 =&
 t \left( 2g(A, {\widetilde B}') + 2(A^1B^2-A^2B^1)A^3 \right)\\
  &+
 t^2 \left(g({\widetilde B}', {\widetilde B}') + 2( A^1B^2-A^2B^1){B^3}' + (A^1B^2-A^2B^1)^2 + (B^2{B^1}' - B^1{B^2}')A^3 \right)\\
  &+
 t^3 (B^2{B^1}' - B^1{B^2}') ({B^3}' + A^1B^2 - A^2B^1)\\
  &+
 t^4 \frac14 (B^2{B^1}' - B^1{B^2}')^2,\\
 g_{12}=&g(f_s, f_t)\\
 =&
 g(A, \widetilde B)
  +
 t (A^1B^2 - A^2B^1)B^3 
  +
 t^2 \frac12 (B^2 {B^1}' - B^1 {B^2}') B^3,\\
 g_{22}=&g(f_t, f_t)=0.
\end{align*}
Obviously,
the degeneracy of the first fundamental form coincides with the vanishing $g_{12}$.
Denoting the second fundamental form $II$ as $II= h_{11}ds^2 + 2h_{12}dsdt + h_{22}dt^2$,
the mean curvature $H$ of $f$ can be given by
\begin{equation*}
 H= -\frac{g_{11}h_{22} -2g_{12}h_{12} } { 2{g_{12}}^2 }.
\end{equation*}
In particular,
we have the equivalence
\begin{equation}\label{minimality1}
 H=0 \Longleftrightarrow g_{11} g ( \nabla_{\partial_t} f_t, f_s \times f_t) -2 g_{12} g ( \nabla_{\partial_s} f_t, f_s \times f_t )=0.
\end{equation}
Straightforward computations show
\begin{equation*}
 \nabla_{\partial_t} f_t
 =
 -B^2B^3 E_1 - B^1 B^3 E_2,
\end{equation*}
\begin{align*}
 f_s \times f_t
 =&
 \left(
   \left(
     A^3 + t({B^3}' + A^1B^2 - A^2 B^1) + \frac{t^2}2 (B^2 {B^1}' - B^1 {B^2}')
   \right)
   B^2
   - (A^2 + t {B^2}') B^3
 \right)
 E_1\\
 &+
 \left(
   \left(
     A^3 + t({B^3}' + A^1B^2 - A^2 B^1) + \frac{t^2}2 (B^2 {B^1}' - B^1 {B^2}')
   \right)
   B^1
   - (A^1 + t {B^1}') B^3
 \right)
 E_2\\
 &+
 \left(
   (A^1 + t{B^1}') B^2 - (A^2 + t {B^2}') B^1
 \right)
 E_3,
\end{align*}
and then we obtain
\begin{equation*}
 g(\nabla_{\partial_t} f_t, f_s \times f_t)
 =
 -g_{12} (B^3)^2.
\end{equation*}
Thus,
since the non-degeneracy of $f$ means that $g_{12}$ vanishes nowhere,
the minimality condition \eqref{minimality1} can be represented as
\begin{equation}\label{minimality2}
 H=0 \Longleftrightarrow g_{11}(B^3)^2 + 2g(\nabla_{\partial_s} f_t, f_s \times f_t) = 0.
\end{equation}
The covariant derivative of $f_t$ with respect to $\partial_s$ is computed as
\begin{align*}
\nabla_{\partial_s} f_t
 =&
 \left(
   {B^1}' - \frac12
   \left(
     A^3 + t({B^3}' + A^1B^2 - A^2 B^1) + \frac{t^2}2 (B^2 {B^1}' - B^1 {B^2}')
   \right)
   B^2
   -\frac12 (A^2 + t {B^2}') B^3
 \right)
 E_1\\
 &+
 \left(
   {B^2}' - \frac12
   \left(
     A^3 + t({B^3}' + A^1B^2 - A^2 B^1) + \frac{t^2}2 (B^2 {B^1}' - B^1 {B^2}')
   \right)
   B^1
   -\frac12 (A^1 + t {B^1}') B^3
 \right)
 E_2\\
 &+
 \left(
   {B^3}' + \frac12 (A^1 + t{B^1}') B^2 - \frac12 (A^2 + t {B^2}') B^1
 \right)
 E_3.
 \end{align*}
Therefore we have
\begin{align}
 g ( \nabla_{\partial_s} f_t, f_s \times f_t )
 =&
 g(A, \widetilde B \times {\widetilde B}') + \frac12 g(A, \widetilde B) g(A, B) \notag\\
 &+
 t
 \begin{pmatrix}
    - (A^1B^2-A^2B^1)  \left( A^3(B^3)^2 + (B^2{B^1}' - B^1 {B^2}') \right) \vspace{5pt}\\
    +\frac12 g(A, \widetilde B) (-B^1 {B^1}' + B^2 {B^2}' -B^3 {B^3}') \vspace{5pt}
 \end{pmatrix} \notag\\
 &+
 t^2
 \begin{pmatrix}
   -\frac12 (B^2 {B^1}' - B^1 {B^2}') \left( A^3(B^3)^2 + (B^2{B^1}' - B^1 {B^2}') \right) \vspace{5pt}\\
   -(A^1B^2-A^2B^1){B^3}' (B^3)^2 \vspace{5pt}\\
    -\frac12 (A^1B^2-A^2B^1)^2 (B^3)^2 \vspace{5pt}
 \end{pmatrix} \label{h_12}\\
 &+
 t^3
 \begin{pmatrix}
   -\frac12 (A^1B^2-A^2B^1) (B^2 {B^1}' - B^1 {B^2}') (B^3)^2 \vspace{5pt}\\
   -\frac12 (B^2 {B^1}' - B^1 {B^2}'){B^3}' (B^3)^2 \vspace{5pt}
 \end{pmatrix} \notag\\
 &+
 t^4
 \left(-\frac18\right) (B^2 {B^1}' - B^1 {B^2}') (B^3)^2.\notag
\end{align}
Thus the minimality of null scrolls \eqref{minimality2} can be converted into the equations of coefficients for $t^k$, $k=0,1,2,3,4$
in $g_{11}(B^3)^2 + 2g(\nabla_{\partial_s} f_t, f_s \times f_t) = 0$.
\begin{Lemma}\label{lemmincondition}
 For any curve $\widetilde B= \sum_{j=1}^{3} B^je_j$ which takes values in the light cone in $\mathfrak{nil}_3$,
 there uniquely exists a function $\beta:I \to \R$ such that $\widetilde B \times {\widetilde B}' = -\beta \widetilde B$.
 Then it follows that
 \begin{equation*}
 g({\widetilde B}', {\widetilde B}')= \beta^2.
 \end{equation*}
\end{Lemma}
\begin{proof}
 A direct computation shows
 $\widetilde B \times {\widetilde B}'$ takes values in the light cone in $\mathfrak{nil}_3$
 and has the product $0$ with $\widetilde B$ with respect to $g$ at each point.
 This implies the existence and uniqueness of a function $\beta: I \to \R$ such that $\widetilde B \times {\widetilde B}' = -\beta \widetilde B$.
 Thus differentiation derives
 \begin{equation*}
  \widetilde B \times {\widetilde B}'' = -\beta' \widetilde B - \beta {\widetilde B}'.
 \end{equation*}
 Therefore we have
 \begin{align*}
  \beta^2 g({\widetilde B}', {\widetilde B}')
  =&
  g(-\beta' \widetilde B - \widetilde B \times {\widetilde B}'', -\beta' \widetilde B - \widetilde B \times {\widetilde B}'')\\
  =&
  g(\widetilde B \times {\widetilde B}'', \widetilde B \times {\widetilde B}'')\\
  =&
  g(\widetilde B, {\widetilde B}'')^2\\
  =&
  \left( g({\widetilde B}', {\widetilde B}') \right)^2.
 \end{align*}
 This implies $g({\widetilde B}', {\widetilde B}') = \beta^2$.
\end{proof}
\begin{Theorem}\label{thm minimal condition}
 Let $\gamma$ be a null curve in $\Nil$,
 $\widetilde B= \sum_{i=1}^3 B^i e_i$ be a curve taking values in the light cone in $\mathfrak{nil}_3$
 and $\gamma (s) \cdot \exp( t \widetilde B(s) )$ be a null scroll.
 Then following statements are equivalent.
 \begin{enumerate}
  \item
  The null scroll $\gamma(s) \cdot \exp( t \widetilde B(s) )$ has the mean curvature $0$.
  \item
   $g (A, \widetilde B)=0$ or $g(A, B) =2\beta$ holds.
 \end{enumerate}
 Here,
 the vector field $B$ is given by
 $B=B^1e_1 + B^2e_2-B^3e_3$
 and
 the function $\beta$ is defined for $\widetilde B$ in Lemma $\ref{lemmincondition}$.
\end{Theorem}
\begin{proof}
The coefficient of $t^0$ for $g_{11}(B^3)^2 + 2g(\nabla_{\partial_s} f_t, f_s \times f_t) $ is given by
\begin{equation}\label{eqt^0part}
 2g (A, \widetilde B \times {\widetilde B}') + g(A, \widetilde B) g(A, B).
\end{equation}
By Lemma \ref{lemmincondition},
\eqref{eqt^0part} is rewritten into
\begin{equation*}
 g(A, \widetilde B) (g(A,B)-2\beta),
\end{equation*}
and then the minimality condition \eqref{minimality2} derives
\begin{equation*}
  g(A, \widetilde B) = 0 \quad \text{or} \quad g(A,B)=2\beta.
 \end{equation*}
 It is easy by \eqref{h_12} to see that the coefficient functions of $t^3$ and $t^4$ satisfy the minimality condition \eqref{minimality2}.
 Let us prove that the coefficient function of $t^2$ always holds the minimality condition \eqref{minimality2}.
 The coefficient of $t^2$ in $g_{11} (B^3)^2 + 2 g (\nabla_{\partial_s} f_t, f_s \times f_t)$ is computed as
 \begin{equation}\label{coefficient of t^2}
  g({\widetilde B}', {\widetilde B}') (B^3)^2 - (B^2 {B^1}' - B^1 {B^2}')^2.
 \end{equation}
 Since $B^2 {B^1}' - B^1 {B^2}' = \beta B^3$ holds,
 by Lemma \ref{lemmincondition} we can see the coefficient \eqref{coefficient of t^2} of $t^2$ always vanishes.
 Next,
 we show that the equation of coefficient of $t$ vanishes when $g (A, \widetilde B)=0$ or $g(A, B) =2\beta$ holds.
 When $g(A, \widetilde B) =0$ holds,
 that is,
 the velocity and the ruling are linearly dependent at each point,
 we have
 \begin{equation*}
  A^1B^2-A^2B^1=0 \quad\text{and}\quad g(A, {\widetilde B}')=0.
 \end{equation*}
 Then it can be seen that the coefficient of $t$ in $g_{11}(B^3)^2 + 2g(\nabla_{\partial_s} f_t, f_s \times f_t)$,
 \begin{equation}\label{tpart}
  g(A, \widetilde B) (-B^1{B^1}' + B^2 {B^2}' - B^3 {B^3}')
   -2(A^1B^2-A^2B^1) (B^2 {B^1}'-B^1 {B^2}')
   +2 g (A, {\widetilde B}') (B^3)^2
 \end{equation}
 vanishes.
 Finally, assume that $g(A, B)=2\beta$ holds.
 If $\beta =0$,
 $B$ and $B'$ are linearly dependent at each point.
 Then it can be seen easily that the coefficient \eqref{tpart} of $t$ vanishes.
 Considering the case of $\beta \neq 0$,
 since $A \times \widetilde B$ can be represented as
 \begin{equation*}
  A \times \widetilde B = \frac{g(A, {\widetilde B}')} {\beta} \widetilde B - \frac{g( A, \widetilde B )} {\beta} {\widetilde B}',
 \end{equation*}
 we have
 \begin{equation*}
  A^1B^2-A^2B^1 = \frac{g(A, {\widetilde B}')} {\beta} B^3 - \frac{g( A, \widetilde B )} {\beta} {B^3}'.
 \end{equation*}
 Thus a simple computation shows \eqref{tpart} vanishes.
 \end{proof}


\subsection{Null frames for null curves in $\Nil$}\label{sec:nullframe}
We give a frame along a null curve in $\Nil$,
which we call a {\it null frame}.
By identifying the Lie algebra $\mathfrak{nil}_3$ with the Minkowski $3$-space $\Min_{(-,+,+)}$ as a vector space,
the curve theory of $\Min_{(-,+,+)}$ can be transported into $\Nil$.
\begin{Lemma}\label{Lorentzplane}
 Let $\langle\ , \ \rangle$ denote the Lorentzian metric in $\Min_{(-,+,+)}$.
 Moreover, fix a null vector $v \in \Min_{(-,+,+)}$.
 For each $2$-dimensional Lorentzian vector subspace $W \subset \Min_{(-,+,+)}$ which includes $v$,
 there exists a unique null vector $w \in W$ which satisfies $\langle v, w \rangle = 1$.
\end{Lemma}
\begin{proof}
 Fix a Lorentzian subspace $W \subset \Min_{(-,+,+)}$ including $v$,
 and take a non-zero spacelike vector $x \in W$.
 Then $v$ and $x$ span $W$.
 Let $w= c_1 v + c_2 x \in W$ be a null vector field which satisfies $\langle v, w \rangle =1$,
 and then
 \begin{equation*}
  0= \langle w, w \rangle = 2c_1 c_2 \langle v, x \rangle + {c_2}^2 \langle x, x \rangle,
 \end{equation*}
 \begin{equation*}
  1= \langle v, w \rangle = c_2 \langle v, x \rangle
 \end{equation*}
 hold.
 Thus coefficients $c_1$ and $c_2$ must be
 \begin{equation*}
  c_1=-\frac{\langle x, x \rangle}{2 \langle v, x \rangle^2}
  \quad \text{and} \quad
  c_2=\frac{1}{\langle v, x \rangle}.
 \end{equation*}
 It can be checked easily that
 the null vector $w$ does not depend on the choice of a spacelike vector $x$ satisfying $W=span\{ v, x \}$.
\end{proof}
Let $\gamma$ be a null curve in $\Nil$ and
denote the bundle along $\gamma$ consisting of the vectors
which have product $0$ with $\frac{d \gamma}{ds}$ by $T^{\perp}\gamma$,
that is,
\begin{equation*}
 T^{\perp} _{\gamma(s)}\gamma = \left\{ v \in T_{\gamma(s)} \Nil \ \middle| \ g \left( v, \frac{d\gamma}{ds}(s) \right)=0 \right\}.
\end{equation*}
Since the subspace $T^{\perp} _{\gamma(s)}\gamma$
is a $2$-dimensional vector space including $T_{\gamma(s)}\gamma$,
the bundle $T^{\perp}\gamma$ can be decomposed into
\begin{equation*}
 T^{\perp}\gamma = T\gamma \oplus S(T^{\perp}\gamma)
\end{equation*}
for some spacelike line bundle $S(T^{\perp}\gamma)$,
called a {\it screen bundle} of $\gamma$ (see \cite{DJ2007}).
Thus each fixed screen bundle of $\gamma$ gives an orthogonal decomposition of $T\gamma$:
\begin{equation*}
 T_{\gamma} \Nil= \left( S(T^{\perp}\gamma) \right)^{\perp} \oplus_{\mbox{orthogonal}} S(T^{\perp}\gamma).
\end{equation*}
Here,
clearly,
the plane bundle $\left( S(T^{\perp}\gamma) \right)^{\perp}$ includes the tangent vector $\frac{d\gamma}{ds}$
and
gives Lorentzian plane at each point of $\gamma$.
Therefore Lemma \ref{Lorentzplane} shows the existence of the $\mathfrak{g}$-valued vector field $B$
satisfying the condition $g(A, B)=1$.
Setting the spacelike vector field $C$ as $C= A \times B$,
we obtain the following proposition.
\begin{Proposition}\label{propnullfrenetframe}
 Every null curve $\gamma(s)$ has a frame $(A, B, C)$ and real-valued functions $k_i \ (i= 0, 1, 2)$ satisfying the following conditions$:$
 \begin{align}
  A&= \gamma^{-1} \frac{d\gamma}{ds},\quad 
  g(A, B) = g(C, C) =1,\notag\\
  g(A, A) &= g(B, B) = g(A, C) = g(B, C) =0,\label{nullfrenetframe}\\
  (A', B'&, C')
  =
  (A, B, C)\begin{pmatrix} k_0&0&-k_2\\0&-k_0&-k_1\\k_1&k_2&0\end{pmatrix}\notag.
 \end{align}
\end{Proposition}
\begin{proof}
 For a null curve $\gamma$,
 take a frame $(A, B, C)$ as discussed above.
 Then it is sufficient to show the frame $(A, B, C)$ satisfies the last condition.
 Let us denote $A'$, $B'$, and $C'$ as
 \begin{equation*}
  A'= a_A A + b_A B + c_A C,
 \quad
  B'= a_B A + b_B B + c_B C,
 \end{equation*}
 and
 \begin{equation*}
  C'= a_C A + b_C B + c_C C.
 \end{equation*}
 Clearly,
 we have $b_A=0$, $a_B=0$, and $c_C=0$ since $g(A, A)$, $g(B, B)$, and $g(C, C)$ are constant.
 Moreover, since $g(A, B)$, $g(A, C)$, and $g(B, C)$ are constant,
 we have
 \begin{equation*}
  a_A = g(A', B) = \frac{d}{ds} g(A, B) -g(A, B') = -g(A, B') = -b_B,
 \end{equation*}
 \begin{equation*}
  c_A = g(A', C) = \frac{d}{ds} g(A, C) -g(A, C') = -g(A, C') = -b_C,
 \end{equation*}
 \begin{equation*}
  c_B = g(B', C) = \frac{d}{ds} g(B, C) -g(B, C') = -g(B, C') = -a_C.
 \end{equation*}
 Therefore we obtain the last condition of the proposition by putting $k_i$ as
 $k_0=a_A=-b_B$, $k_1=c_A=-b_C$, and $k_2=c_B=-a_C$.
\end{proof}
\begin{Remark}
 If necessary,
 by reparametrizing the curve,
 we can take $k_0=0$ from the beginning \cite[Lemma $2.1$ and Remark $2.3$]{BD1996}.
 In the case of $\Min$,
 such a parameter is called the {\it distinguished parameter}.
 We would like to note that the frame $(A, B, C)$ depends on the choice of a screen bundle.
 Thus another choice of a screen bundle gives a change in $k_i$. 
\end{Remark}
We call a frame $(A, B, C)$ given by Proposition \ref{propnullfrenetframe} with $k_0=0$
a {\it null frame} for a null curve $\gamma$
and
the functions $k_1$ and $k_2$ the {\it first curvature} and {\it second curvature} of $\gamma$ with respect to $(A, B, C)$.
\begin{Theorem}\label{curvefromcurvatures}
 For any real-valued functions $k_1$ and $k_2$,
 there exists a null curve $\gamma$ in $\Nil$ which has $k_i \ (i=1, 2)$ as the first and second curvature
 with respect to some null frame. 
\end{Theorem}
\begin{proof}
 For a null frame $(\sum_{i=1}^3 A^ie_i, \sum_{i=1}^3 B^ie_i, \sum_{i=1}^3 C^ie_i)$,
 let us contemplate the matrix
 \begin{equation*}
  \begin{pmatrix}
   A^1&B^1&C^1\\A^2&B^2&C^2\\A^3&B^3&C^3
  \end{pmatrix}
 \end{equation*}
 and designate it as $(A, B, C)$ herein.
 Moreover, define a matrix $F_0$ by
 \begin{equation*}
  F_0=
  \begin{pmatrix}
   -1/{\sqrt 2} & 1/{\sqrt 2} & 0\\ 1/{\sqrt 2} & 1/{\sqrt 2} & 0\\ 0&0&1
  \end{pmatrix}.
 \end{equation*}
 For any null frame $(A, B, C)$,
 there exists a unique $\text{O}(2,1)$-valued smooth map $X$ such that $(A, B, C)= XF_0$.
 Conversely,
 for any $\text{O}(2,1)$-valued smooth map $X$,
 the map $XF_0$ satisfies the conditions \eqref{nullfrenetframe}.
 Thus we should show the proposition with the initial condition
 \begin{equation}\label{initial}
  \gamma(0)=(0, 0, 0)
  \quad \text{and} \quad
  \left( A(0), B(0), C(0) \right) = F_0.
 \end{equation}
 First,
 take the solution of the following differential equations with \eqref{initial}
 \begin{equation}\label{system}
  \frac{d}{ds}\begin{pmatrix}A^1\\A^2\\A^3\end{pmatrix}
  =
  k_1\begin{pmatrix}C^1\\C^2\\C^3\end{pmatrix},
  \quad
  \frac{d}{ds}\begin{pmatrix}B^1\\B^2\\B^3\end{pmatrix}
  =
  k_2\begin{pmatrix}C^1\\C^2\\C^3\end{pmatrix},
  \quad
  \frac{d}{ds}\begin{pmatrix}C^1\\C^2\\C^3\end{pmatrix}
  =
  -k_2\begin{pmatrix}A^1\\A^2\\A^3\end{pmatrix}
  -k_1\begin{pmatrix}B^1\\B^2\\B^3\end{pmatrix}.
 \end{equation}
 It is only necessary to prove that the solution $(A, B, C)$ satisfies the conditions \eqref{nullfrenetframe}.
 By direct calculations using \eqref{system},
 we obtain
 \begin{equation*}
  \frac{d}{ds}
   \left(
    A^i B^j +A^j B^i + C^i C^j
   \right)
   =0.
 \end{equation*}
 Then,
 substituting $s=0$,
 we get
 \begin{equation*}
  A^iB^j + A^j B^i + C^i C^j
  =
  \left\{
   \begin{matrix}
    -1 &(i=j=1) \\ 1 &(i=j \neq 1) \\ 0 &(i \neq j)
   \end{matrix}.
  \right.
 \end{equation*}
 Define vector fields $V_1, V_2, V_3$ by
 \begin{equation*}
  V_1= \frac{1}{\sqrt 2} (B-A), \quad V_2=\frac{1}{\sqrt 2} (B+A), \quad V_3= C.
 \end{equation*}
 Consequently, $(V_1, V_2, V_3)$ constitutes an orthonormal basis of $\mathfrak{nil}_3$ at each point,
 signifying the fulfillment of conditions \eqref{nullfrenetframe} by $(A, B, C)$.
 Furthermore, the null curve $\gamma=(\gamma^1, \gamma^2, \gamma^3)$
 which has the null frame $(A, B, C)$
 is given by
 \begin{equation*}
  \gamma^1=\int_0^s A^1 ds,
  \quad
  \gamma^2=\int_0^s A^2 ds,
  \quad
  \gamma^3=\int_0^s (A^3-\frac12 \gamma^2 A^1 + \frac12 A^2 \gamma^1)ds.
 \end{equation*}
\end{proof}
\begin{Theorem}\label{curvaturetoscroll}
 For a null curve $\gamma$ in $\Nil$ with a null frame $(A, B, C)$,
 let $f: \D \to \Nil$ be a null scroll defined on an open subset $\D \subset \C$,
 over $\gamma$ with the ruling $\widetilde B= \sum_{i=1}^3 B^ie_i$ where $B=B^1e_1+B^2e_2-B^3e_3$,
 that is,
 \begin{equation*}
  f(s,t)=\gamma(s) \cdot \exp (t \widetilde B(s)).
 \end{equation*}
 If $k_2=1/2$, then $f$ is minimal.
\end{Theorem}
\begin{proof}
 Since the vector field $B$ satisfies the null frame condition \eqref{nullfrenetframe},
 we can get $\beta = k_2$ where $\beta$ is defined in Lemma \ref{lemmincondition}.
 Then assuming $k_2= 1/2$ derives $2\beta = 1= g(A, B)$.
 Thus Theorem \ref{thm minimal condition} implies $f$ is minimal.
\end{proof}
\noindent
By Theorem \ref{curvefromcurvatures} and Theorem \ref{curvaturetoscroll}
we obtain the following corollary immediately.
\begin{Corollary}\label{Corcurvature}
 For an arbitrary real valued function $k_1$,
 there exists a minimal null scroll such that
 the base curve has the first curvature $k_1$ with respect to some null frame.
\end{Corollary}
In general,
an another coordinate $(x, y)$ for a null scroll $\gamma(s) \cdot \exp(t \widetilde B(s))$ is null coordinate
if and only if
the condition
\begin{equation*}
  s_x =0 \quad \text{and} \quad g_{11} s_y + 2g_{12}t_y =0,
\end{equation*}
or
\begin{equation}\label{conformalcondition}
  s_y =0 \quad \text{and} \quad g_{11} s_x + 2g_{12}t_x =0
\end{equation}
is satisfied.
Let us compute the support function $h$ and the Abresch-Rosenberg differential $Qdz^2$
for the minimal null scroll defined in Theorem \ref{curvaturetoscroll}.
For simplicity,
we assume that $(x, y)$ is a null coordinate system which satisfies the condition \eqref{conformalcondition}.

Since the para-complex number $l=(1+\ip)/2$ has the properties $l^2=l$ and $l \bar l=0$,
a coordinate $z= lx + \bar l y$ for a null coordinate $(x, y)$ gives a conformal coordinate,
and then
the first fundamental form $I$
and the unit normal field $N$
of the null scrolls are represented as
\[I= e^udzd\bar z= e^u dxdy= 2g_{12}s_x t_y dxdy,\]
\[N=-\ip \frac{f_z \times f_{\bar z}} {|g(f_z \times f_{\bar z}, f_z \times f_{\bar z})|^{1/2}}
    =-(g_{12})^{-1} f_s \times f_t.\]
Thus we have the support function
\begin{align*}
 h=&g(-e^{u/2} N, E_3) \\
  =&g( \sqrt{2} |s_x t_y|^{1/2} |g_{12}|^{-1/2} f_s \times f_t, E_3 ).
\end{align*}
Since
$B\times B'=\frac12 B$ and $g(A, B)=1$ hold
for the minimal null scrolls in Theorem \ref{curvaturetoscroll},
$g(f_s \times f_t, E_3)$ and $g_{12}$ are computed as
\begin{align*}
 g(f_s \times f_t, E_3)
 =&
 (A^1+t{B^1}')B^2- (A^2+t{B^2}')B^1\\
 =&
 A^1B^2-A^2B^1+t \frac12 B^3,
\end{align*}
\begin{align*}
 g_{12}
 =&
 g(A, \widetilde B)+ t(A^1B^2-A^2B^1)B^3 + t^2 \frac12 (B^2 {B^1}' - B^1 {B^2}')B^3\\
 =&
 g(A, \widetilde B)- (A^1B^2-A^2B^1)^2 + (A^1B^2-A^2B^1+t \frac12 B^3)^2\\
 =&
 g(A, \widetilde B)-g(A, B) g(A, \widetilde B)+ (A^1B^2-A^2B^1+t \frac12 B^3)^2\\
 =&
 (A^1B^2-A^2B^1+t \frac12 B^3)^2.
\end{align*}
Hence we obtain the support function
\begin{align}
h
=&
\sqrt{2}|s_x t_y|^{1/2} |g_{12}|^{-1/2} g(f_s \times f_t, E_3)\notag\\
=&
\sqrt{2} \epsilon |s_x t_y|^{1/2}\label{supportfctofframednullscroll}
\end{align}
where $\epsilon = {\rm sgn} (g(f_s \times f_t, E_3)) \in \{ \pm1 \}$.

To obtain the Abresch-Rosenberg differential,
we compute $\phi^3=g(f_z, E_3)$ and $g(\nabla_{\partial} f_z, N)$.
Since the derivative $f_z$ of $f$ with respect to $z$ is
\begin{align*}
 f_z
 =&
 l(s_x f_s + t_x f_t) + \bar l (t_y f_t)\\
 =&
 \left( l(s_x (A^1+ t{B^1}') + t_x B^1 ) + \bar l (t_y B^1) \right) E_1
 +
 \left( l(s_x (A^2+ t{B^2}') + t_x B^2 ) + \bar l (t_y B^2) \right) E_2\\
 &+
 \left( l(s_x D^3 + t_x B^3 ) + \bar l (t_y B^3) \right) E_3,
\end{align*}
we have
\begin{equation*}
 \phi^3= l(s_x D^3 + t_x B^3 ) + \bar l (t_y B^3)
\end{equation*}
where $D^3= g(f_s, E_3) = A^3 + t({B^3}' + A^1B^2-A^2B^1) + t^2 \frac12 (B^2 {B^1}'- B^1 {B^2}')$.
Besides,
a straightforward computation shows
\[
 \nabla_{\partial} f_z
 =
 l \left(
  s_{xx} f_s + t_{xx} f_t
  + (s_x)^2 \nabla_{\partial_s} f_s + 2s_xt_x \nabla_{\partial_t} f_s + (t_x)^2 \nabla_{\partial_t} f_t
 \right)
 +
 \bar l \left(
  t_{yy} f_t + (t_y)^2 \nabla_{\partial_t} f_t
 \right),
\]
and then,
by using the conditions \eqref{minimality1}, \eqref{minimality2}, and \eqref{conformalcondition},
the Hopf differential
can be computed as
\[
g(\nabla_{\partial} f_z, N)
=
l \left(
 (s_x)^2 h_{11} - (t_x)^2 (B^3)^2
\right)
+
\bar l \left(
 (t_y)^2 (B^3)^2
\right).
\]
Hence the Abresch-Rosenberg differential of null scrolls defined in Theorem \ref{curvaturetoscroll}
can be written as
\begin{equation}\label{ARdiffofframedscroll}
 Q
 =
 \frac{l}4 \left(
  -(s_x)^2 h_{11} - (s_x)^2 (D^3)^2 - 2s_xt_x D^3 B^3
 \right).
\end{equation}
Using the null frame conditions \eqref{nullfrenetframe},
the Abresch-Rosenberg differential \eqref{ARdiffofframedscroll} is computed as
\begin{align}
 Q
 =&
 \frac{l}{4} (s_x)^2 (g_{12})^{-1} \left(k_1 g_{12} + \frac{t^2}{8} (1+ 2A^3B^3 - (A^1B^2-A^2B^1)^2)\right)\notag\\
 =&
 \frac{l}{4} (s_x)^2 k_1.
\end{align}
We would like to note that
the calculations of the Abresch-Rosenberg differential until \eqref{ARdiffofframedscroll}
do not use the null frame conditions \eqref {nullfrenetframe}.
This means that the Abresch-Rosenberg differential $Qdz^2$ of minimal null scrolls
(which are not necessarily defined in Theorem \ref{curvaturetoscroll})
satisfies $Q \overline Q=0$.
Moreover,
the solution $w$ of the integrability conditions \eqref{integrablecondition} satisfies 
\[e^w = \frac1{16} h^2 = \frac18 s_xt_y.\]
On the other hand,
the exact solution $w$ of the Liouville equation
\[\frac12 w_{z \bar z} + e^w =0\]
is given by
\[w= \log \left( - \frac{p_xq_y}{(p(x)+q(y))^2} \right), \quad p(x)+q(y) \neq 0, \ p_x q_y<0\]
where $p$ and $q$ are functions of $x$ and $y$, respectively.
Therefore the coordinate systems $(s, t)$ and $(x, y)$ satisfy the equation
\[\frac{s_x t_y}{8} = -\frac{p_x q_y}{(p+q)^2}.\]
Hence we obtain the explicit representation of null scroll's coordinate system $(s, t)$
in terms of null coordinate system $(x, y)$ by separating variables
\[s(x)=8p(x), \quad t= \frac1{p(x)+q(y)}.\]

Furthermore,
for minimal null scrolls defined in Theorem \ref{curvaturetoscroll},
another minimal null scroll with the same or multiplied by $-1$ support function and the same Abresch-Rosenberg differential
can be obtained by the following theorem.
\begin{Theorem}\label{uniquenullscroll}
Let $f_1$ be a minimal null scroll defined in Theorem \ref{curvaturetoscroll}
and
$F_0$ be an element of $SO(2,1)$.
Denote the support function and the Abresch-Rosenberg differential of $f_1$ by $h$ and $Qdz^2$
and define a timelike surface $f_2$ so that
\begin{equation}\label{eqnonuniqueness1} 
 (\phi_2^1, \phi_2^2, \ip \phi_2^3)
 =
 (\phi_1^1, \phi_1^2, \ip \phi_1^3) F_0
\end{equation}
where $\phi_k^j$ for $j=1,2,3$ and $k=1,2$ are given by ${f_k}_z= \sum_{j=1}^{3} \phi_k^jE_j$ for $k=1,2$.
Then $f_2$ has the support function $\pm h$ and the Abresch-Rosenberg differential $Qdz^2$,
and it is also a minimal null scroll.
\end{Theorem}
\begin{proof}
 Denote the surface in $\Min_{(-,+,+)}$ induced from $f_1$ by $\Phi_1$, and define another surface $\Phi_2$ as
 \begin{equation}\label{eqnonuniqueness2}
  {\Phi_2}_z = {\Phi_1}_z F_0.
 \end{equation}
 Then $\Phi_1$ and $\Phi_2$ differ by an isometry in $\Min_{(-,+,+)}$.
 Therefore $\Phi_2$ has constant mean curvature $1/2$, and then there exists a timelike minimal surface $\tilde f$ in $\Nil$ which has the support function $\pm h$ and the Abresch-Rosenberg differential $Qdz^2$ by Theorem \ref{thmnonuniqueness}.
 Since \eqref{eqnonuniqueness1} and \eqref{eqnonuniqueness2} imply that $f_2$ and $\tilde f$ differ by a left translation, $f_2$ has the support function $\pm h$ and the Abresch-Rosenberg differential $Qdz^2$.

 To complete the proof, we need only show that the timelike minimal surface $\tilde f$ becomes a null scroll.
 Since $\Phi_1$ is a $B$-scroll, $\Phi_2$ is also a $B$-scroll due to Remark \ref{rem:B-scrolls}.
 Thus the discussion in subsection \ref{subsec:Bscrollnullscroll} shows that $\tilde f$ is a null scroll in $\Nil$.
 Consiquently timelike surface $f_2$ is minimal null scroll and has the support function $\pm h$ and the Abresch-Rosenberg differential $Qdz^2$
\end{proof}
From the above discussion,
we can show the following theorem.
\begin{Theorem}\label{thmmain}
 If a null scroll $f$ is minimal,
 then the Abresch-Rosenberg differential $Qdz^2$ of $f$ satisfies $Q \overline Q=0$.
 Conversely,
 every timelike minimal surface with $Q \overline Q=0$ is of the form 
 \begin{equation*}
  \gamma(s) \cdot \exp (t \widetilde B(s))
 \end{equation*}
 where $\gamma$ is a null curve and $\widetilde B$ is a curve
 which takes values in the light cone in $\mathfrak{nil}_3$.
\end{Theorem}
\begin{proof}
 The first half of the claim is already proved.
 We prove the latter half.
 Let $f$ be a timelike minimal surface of the Abresch-Rosenberg differential $Qdz^2$ with $Q \overline Q=0$.
 We show that $f$ is a null scroll.
 Let us denote the support function of $f$ by $h$.
 It is only necessary to construct a minimal null scroll which has the support function $\pm h$ and the Abresch-Rosenberg differential $Qdz^2$
 because we have Theorem \ref{thmnonuniqueness} and Theorem \ref{uniquenullscroll}.

 For null scrolls $\gamma(s) \cdot \exp(t\widetilde B(s))$,
 $(A, B, A \times B)$ are not always frames.
 It is easy to check that
 $(A, B, C)$ for $\gamma(s) \cdot \exp(t\widetilde B(s))$ is not a frame
 if and only if
 $A$ and $B$ are linearly dependent.
 If $g(A, B)=0$ holds, minimal null scroll $\gamma(s) \cdot \exp(t\widetilde B(s))$ must be a vertical plane
 (see Proposition \ref{betaeq0} and Example \ref{exampleverticalplane}).
 The Abresch-Rosenberg differential of vertical plane is $0$. 
 From now on we assume that $f$ is not a vertical plane,
 and construct minimal null scrolls from the first curvature
 by Corollary \ref{Corcurvature}.
 The holomorphicity of $Q$ means that $Q$ can be separated into functions of $x$ and $y$,
 \begin{equation*}
  Q= l S(x) + \overline l T(y)
 \end{equation*}
 where $z= lx + \overline l y$, and $S$ and $T$ are real valued functions of $x$ and $y$, respectively.
 Then the condition $Q  \overline Q=0$ implies
 \begin{equation*}
  S=0 \quad \text{or} \quad T=0.
 \end{equation*}
 We prove the case of $T=0$.
 The solutions of the integrability conditions \eqref{integrablecondition} are given by
 \begin{equation*}
  w= \log \left( -\frac{p_x q_y} {(p(x)+q(y))^2} \right)
 \end{equation*}
 where $p$ and $q$ are functions of $x$ and $y$, respectively, such that $p(x)+q(y) \neq 0$ and $p_x q_y<0$.
 Let us define functions $k_1(s)$ and $k_2(s)$ by
 \begin{equation}\label{ARdiffcurvature}
  k_1(s(x))
  =
  \frac{S(x)}{16{p_x}^2},
  \quad
  k_2(s(x))
  =
  \frac12
 \end{equation}
 where the parameter $s$ is defined by $s(x)=8p(x)$. 
 Then from Theorem \ref{curvefromcurvatures}
 we can obtain a null frame $(A, B, C)$ and a null curve $\gamma(s)$
 which has the first and second curvature $k_1(s)$ and $k_2(s)$ with respect to $(A, B, C)$.
 Now, we consider the map
 \begin{equation}\label{nullscrollxy}
  \gamma(s(x)) \cdot \exp \left( \frac1{p(x) + q(y)} \widetilde B(s(x)) \right).
 \end{equation}
 Here,
 $\widetilde B= \sum_{i=1}^3 B^ie_i$ is the curve in $\mathfrak{nil}_3$
 determined from $B=B^1e_1 + B^2e_2 - B^3e_3$.
 Direct computations show that for \eqref{nullscrollxy},
 the Abresch-Rosenberg differential is $l S(x) dz^2$
 and the support function can be represented as $4 {\rm sgn}(g(f_s \times f_t, E_3)) e^{w/2}$.
 Since the null scroll \eqref{nullscrollxy}
 has the support function $4e^{w/2}$ or $-4e^{w/2}$ and shares the Abresch-Rosenberg differential with $f$,
 Theorem \ref{uniquenullscroll} shows timelike minimal surface $f$ has to be a null scroll.
\end{proof}
\begin{Example}\label{example_horizontalumbrella}
 Let us construct timelike minimal surfaces
 with vanishing Abresch-Rosenberg differential
 except for vertical planes.
 Solve the differential equation \eqref{system}
 under the condition $k_1=0$, $k_2=1/2$ and the initial conditions $A^i(0)=A_0^i$, $B^i(0)=B_0^i$, and $C^i(0)=C_0^i$.
 Then the null frame is given by
 \begin{equation*}
  A^i(s)=A_0^i,
  \quad
  B^i(s)=-\frac{s^2}8 A_0^i + \frac s2 C_0^i + B_0^i,
  \quad
  C^i(s)=-\frac s2 A_0^i +C_0^i,
 \end{equation*}
 and the base curve $\gamma$ is given by $\gamma(s)=(A_0^1 s, A_0^2 s, A_0^3 s)$.
 Hence the null scroll $\gamma(s) \cdot \exp(t\widetilde B(s))$
 obtained according to Theorem \ref{curvaturetoscroll} from $k_1=0$ and $k_2=1/2$ is written explicitly
 \begin{align*}
  (A_0^1s, \ A_0^2 s, \ A_0^3 s)
  \cdot
  (t(-\frac{s^2}8 A_0^1 + \frac s2 C_0^1 + B_0^1),
    \ t(-\frac{s^2}8 A_0^2 + \frac s2 C_0^2 + B_0^2),
    \ t(\frac{s^2}8 A_0^3 - \frac s2 C_0^3 - B_0^3))\\
  =
  \left(
    (s-\frac{s^2 t}8)A_0^1 +tB_0^1 + \frac{st}2 C_0^1,
    \ (s-\frac{s^2 t}8)A_0^2 +tB_0^2 + \frac{st}2 C_0^2,
    \ (s-\frac{s^2 t}8)A_0^3 -tB_0^3
  \right).
 \end{align*}
 Since, by $C \times A=A$ and $B \times C=B$, a simple computation shows
 \begin{equation*}
  \det
  \begin{pmatrix}
   A_0^1&B_0^1&C_0^1\\ A_0^2&B_0^2&C_0^2\\ A_0^3&-B_0^3&0
  \end{pmatrix}
  =0,
 \end{equation*}
 the vectors $(A_0^1, A_0^2, A_0^3)$, $(B_0^1, B_0^2, -B_0^3)$, and $(C_0^1, C_0^2, 0)$ are linearly dependent at each point.
 Therefore the minimal null scroll given from $k_1=0$ lies on a plane in $\mathbb{R}^3$ (Figure $1$).
 The surfaces given in Example \ref{example_horizontalumbrella} are called {\it horizontal umbrellas}.
 The condition $k_1=0$ means the vanishing of the Abresch-Rosenberg differential.
 \begin{figure}[t]\label{ex_horizontalumbrella}
 \centering
 \includegraphics[width=6cm]{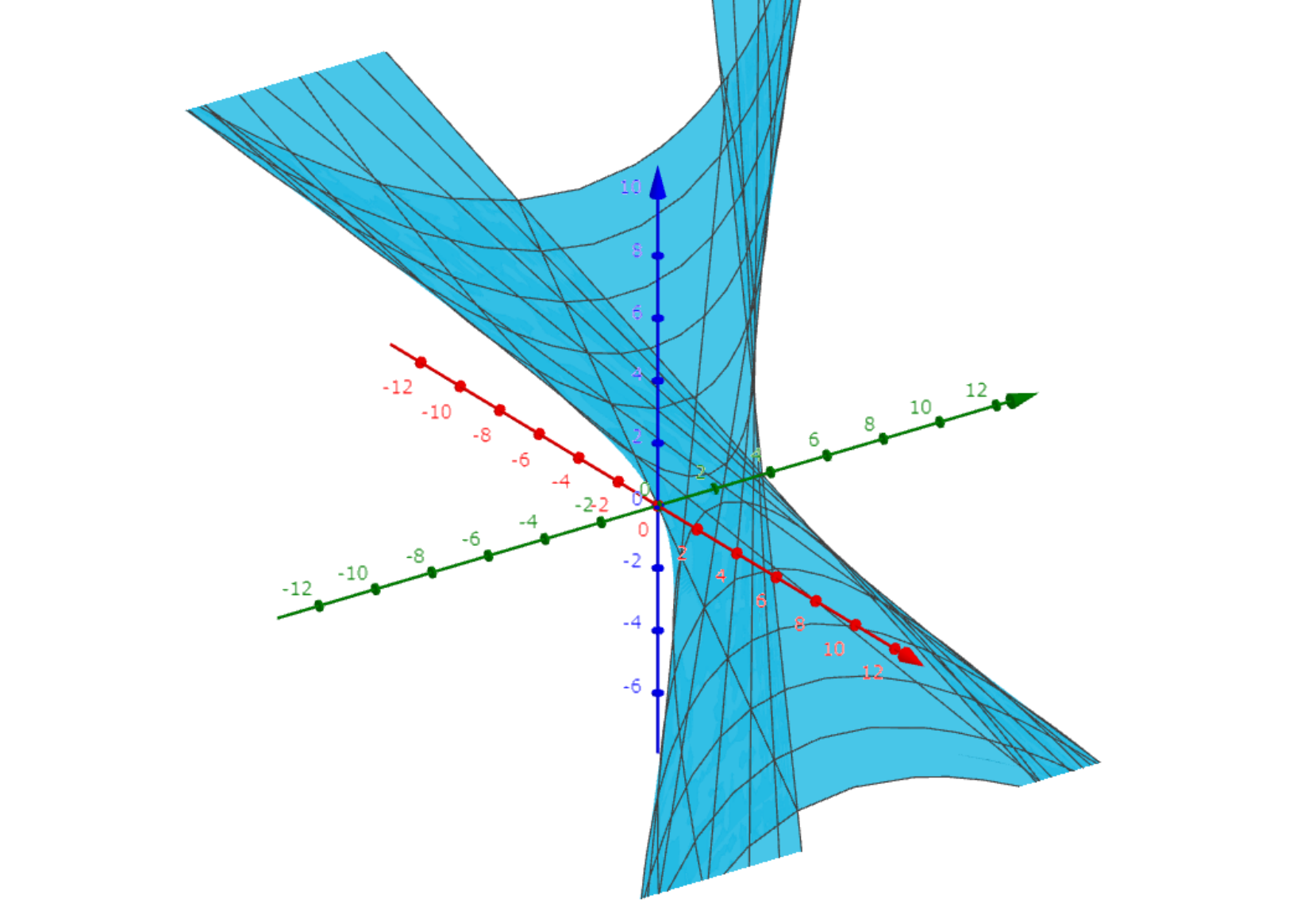}
 \caption{Horizontal umbrella: Timelike minimal surface with the Abresch-Rosenberg differential vanishing anywhere.
 They are included in planes.}
 \end{figure}
\end{Example}


\section{Construction of minimal null scrolls with prescribed rulings}\label{sec:construct}
In Section $\ref{sec:characterization}$,
we observed that timelike minimal surfaces with $Q \overline Q=0$ are obtained as minimal null scrolls,
in particular as null scrolls constructed from null curves with null frames and the second curvatures $k_2=1/2$.
However,
the construction from curvatures is troublesome because of the need to solve a system of differential equations.
In this section,
we will construct the minimal null scrolls with prescribed rulings
using only elementary computations.


\subsection{Construction from the condition $g(A, B)=2\beta$ and $\beta=0$}
It follows from Theorem \ref{thm minimal condition} that
we need to consider only two cases $g(A,\widetilde B)=0$ and $g(A, B)=2\beta$
for constructing minimal null scrolls.
First,
we consider the case of $g(A, B)= 2\beta$.
The following proposition gives the minimal null scrolls
satisfying the minimality condition $g(A, B)=2\beta$ and $\beta=0$
as a part of vertical planes.
\begin{Proposition}\label{betaeq0}
 Let $\gamma$ be an affine null line in $\Nil$
 and
 $\widetilde B= \sum_{i=1}^3 B^ie_i$ be a curve
 which takes values in the light cone in $\mathfrak{nil}_3$,
 satisfies $B^3 \neq 0$,
 and induces the null vector $B=B^1e_1+B^2e_2-B^3e_3$ linearly dependent to $\gamma^{-1} \frac{d\gamma}{ds}$.
 Then the null scroll $\gamma(s) \cdot \exp (t \widetilde B(s) )$ is minimal
 and satisfies the minimality condition $g(A, B)=2 \beta$ with $\beta = 0$.
 Conversely,
 if a minimal null scroll has a ruling $\widetilde B$ with $\beta=0$ and satisfies the minimality condition $g(A, B)=2\beta$,
 the base curve is an affine null line and the ruling $\widetilde B$ induces a null vector field $B$ linearly dependent on the velocity.
\end{Proposition}
\begin{proof}
 First,
 we take the minimal null scroll given by
 an affine null line
 $\gamma(s) = K_1(s)(c^1, c^2, c^3)$
 and a ruling $\widetilde B=\sum_{i=1}^3 B^ie_i$
 which induces the null vector fields $B=B^1e_1+B^2e_2-B^3e_3$
 linearly dependent to the velocity $\gamma^{-1} \frac{d \gamma}{ds}$,
 that is,
 \begin{equation*}
  g(\gamma^{-1} \frac{d \gamma}{ds}, B) = \frac{dK_1}{ds}\left( -c^1B^1+ c^2B^2- c^3B^3 \right) =0.
 \end{equation*}
 Here $K_1$ is a function
 and $c^j \ (j=1,2,3)$ are constants that are not $0$ simultaneously
 and satisfy $-(c^1)^2 + (c^2)^2 + (c^3)^2 =0$.
 By the linear dependence,
 the vector field $B$ can be rewritten as
 $B=K_2 (c^1e_1 + c^2e_2 - c^3e_3)$
 for some smooth function $K_2$.
 Then a simple computation shows
 \begin{equation*}
  B \times B'=0,
 \end{equation*}
 that means $\beta=0$.
 Hence the null scroll  $\gamma(s) \cdot \exp (t \widetilde B(s) )$ fulfills the minimality condition
 $g(A, B)=2\beta$ with $\beta=0$.
 Conversely, we assume $\beta=0$.
 In this case,
 the vector field ${\widetilde B}'$ is null.
 Moreover the condition
 \begin{equation*}
 \widetilde B \times {\widetilde B}' = 0
 \end{equation*}
 means the vector field ${\widetilde B}'$ is linearly dependent to $\widetilde B$ at each point,
 and then
 it follows the differential equations
 \begin{equation*}
  {B^j}' = k B^j \quad j=1,2,3
 \end{equation*}
 for some function $k$.
 Hence the ruling $\widetilde B$ has to be of the form
 \begin{equation}\label{1B}
  \widetilde B(s) = e^{ \int k ds } \left( \sum c^je_j \right)
 \end{equation}
 where $c^j \ (j=1,2,3)$ are some constants.
 Thus,
 by the minimality condition $g(A, B)=2 \beta$ with $\beta=0$,
 we obtain the base curve $\gamma$ with the initial condition $\gamma(0)=(0,0,0)$
 explicitly
 \begin{align}\label{1A}
  \gamma^{-1}\frac{d\gamma}{ds}(s)&= h(s) e^{ \int k ds } \left( c^1e_1 + c^2e_2 - c^3e_3 \right),\\
  \gamma(s)&= \int_0^s\left( h(s) e^{ \int k d s } \right) ds (c^1, c^2, -c^3)\notag
 \end{align}
 for some function $h$.
 Therefore the base curve $\gamma$ draws an affine null line.
\end{proof}
\begin{Remark}
 The first and second fundamental forms of minimal null scrolls
 defined from \eqref{1B} and \eqref{1A}
 can be computed as
 \begin{equation*}
  I=-2g_{12} \left( tk(s)ds^2 +dsdt \right),
 \end{equation*}
 \begin{equation*}
  II=\frac{ g_{12} }{ |g_{12}| } e^{2\int k(s)ds} (c^3)^2 \left( \frac{( h(s)+tk(s) ) ( -h(s) + tk(s) )} {h(s)} ds^2 + 2tk(s) dsdt + dt^2 \right)
 \end{equation*}
 where $g_{12}=-2h(s)e^{2\int k(s)ds} (c^3)^2$.
 Hence the non-degeneracy of $I$ means vanishing $c^3$ nowhere on the domain,
 and the first fundamental form degenerates if and only if the second fundamental form vanishes.
 Furthermore if the first fundamental form is non-degenerate then the second fundamental form is also non-degenerate.
\end{Remark}
Proposition \ref{betaeq0} implies that
minimal null scrolls,
satisfying the minimality condition $g(A, B)=2\beta$ and $\beta=0$,
can be constructed
from an arbitrary null vector field $B$ which has $B^3 \neq 0$ and $\beta=0$,
by defining the velocity $A$ of the base curves as $A=hB$.
A direct computation shows the vector field $E_3$ is tangent to these surfaces,
and then they are part of vertical surfaces,
that is,
Hopf cylinders.
Therefore these minimal null scrolls contain the affine lines in $x_3$-axis direction.
Furthermore,
since the direction of the ruling \eqref{1B} is independent to the parameter $s$,
these minimal null scrolls are part of planes spanned by $(c^1, c^2, c^3)$ and $(0, 0, 1)$.
\begin{Example}[vertical plane]\label{exampleverticalplane}
 Let $\theta \in (0, \pi)$ be a constant.
 Take a null vector field $\widetilde B$ as $\widetilde B(s) = s  e_1 + s\cos\theta e_2 + s \sin\theta e_3$,
 and define a null vector field $A$ as $A= e_1 +\cos\theta e_2 -\sin\theta e_3$.
 Then the null scroll
 \begin{equation*}
  \gamma(s) \cdot \exp(t \widetilde B(s))
  =
  \left(
    (1+t)s,
    (1+t)s \cos\theta,
    (1+t)s \sin\theta-2s\sin\theta
  \right)
 \end{equation*}
 is a timelike minimal surface (Figure 2).
 \begin{figure}[t]
 \centering
 \includegraphics[width=6cm]{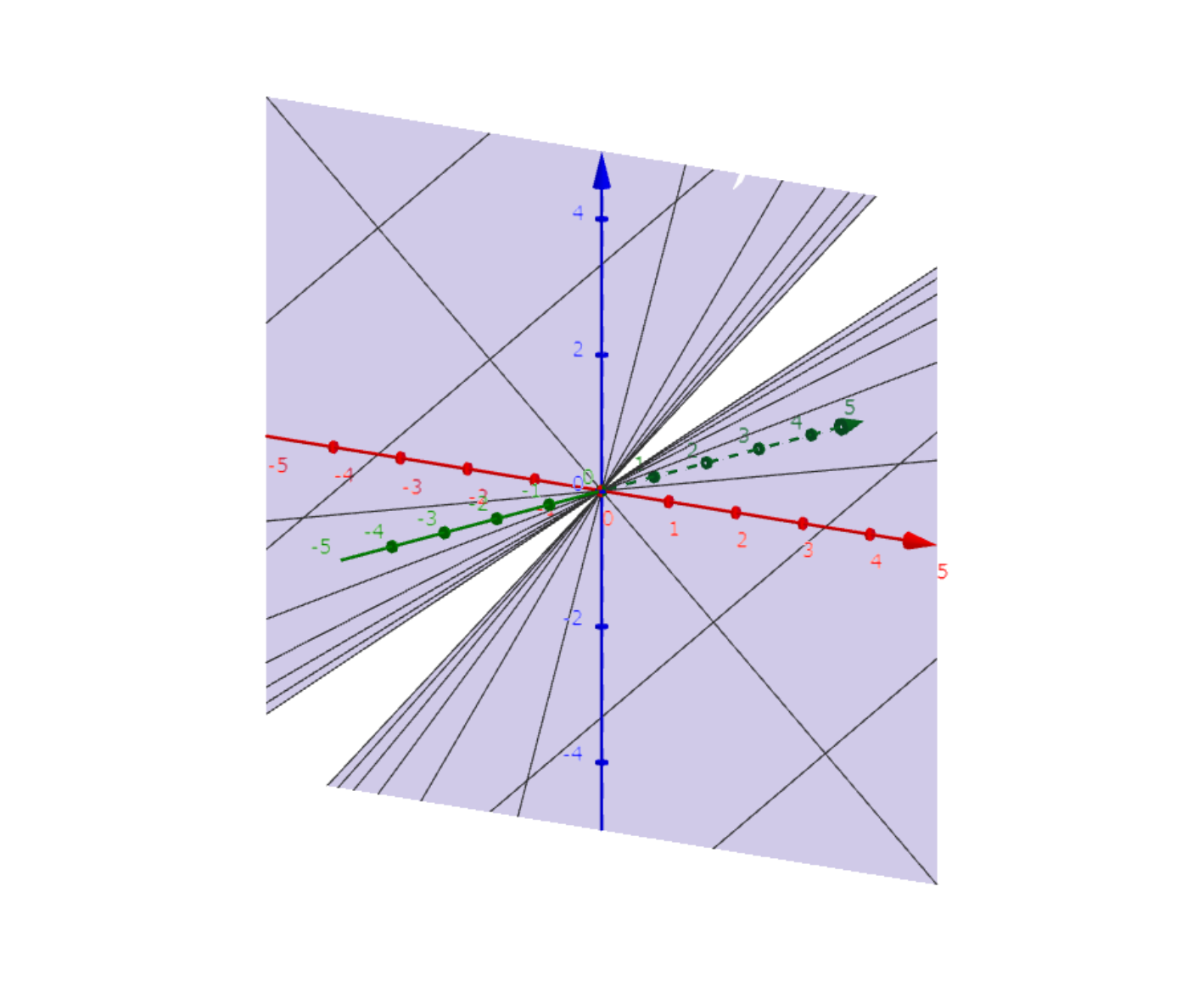}
 \caption{Vertical plane: Minimal null scrolls with the minimality condition $g(A, B)=2\beta$ and $\beta=0$.}
 \end{figure}
\end{Example}


\subsection{Construction from the condition $g(A, B ) = 2\beta$ and $\beta \neq0$}
From now on,
we will construct the minimal null scrolls which satisfy the minimality condition $g(A, B)=2\beta$
with $\beta$ vanishing nowhere.
When $\beta$ does not vanish anywhere,
if necessary replacing $\widetilde B$ with $\frac{1}{2\beta}\widetilde B$,
without loss of generality, we can regard $\beta$ as $1/2$.
For a curve which takes values in the light cone in $\mathfrak{nil}_3$
with $\beta=1/2$,
the base curve $\gamma$ can be constructed as follows.

Since for a null scroll $\gamma(s) \cdot \exp( t\widetilde B(s) )$,
the vector field $\gamma^{-1}\frac{d\gamma}{ds} \times B$ belongs to $B^\perp$ at each point,
it can be represented as
\begin{equation*}
 \gamma^{-1}\frac{d\gamma}{ds} \times B= a B' + b B
\end{equation*}
for some functions $a$ and $b$.
Therefore when the minimality condition $g(\gamma^{-1}\frac{d\gamma}{ds}, B)=2\beta$ with $\beta=1/2$ holds,
the equations
\begin{equation*}
 g(\gamma^{-1}\frac{d\gamma}{ds} \times B, \gamma^{-1}\frac{d\gamma}{ds})=0
 \quad \text{and} \quad
 g(\gamma^{-1}\frac{d\gamma}{ds} \times B, B')= g(\gamma^{-1}\frac{d\gamma}{ds}, B \times B')=1/2
\end{equation*}
derive $a=2$ and $b=-2g(\gamma^{-1}\frac{d\gamma}{ds}, B')$.
Thus it can be seen that for any curve $\widetilde B$ in the light cone in $\mathfrak{nil}_3$ except for the origin,
determining the direction for $\gamma^{-1}\frac{d\gamma}{ds} \times B$
is equivalent to
taking a function $g(\gamma^{-1}\frac{d\gamma}{ds}, B')$.
On the other hand,
by Lemma \ref{Lorentzplane},
taking a null vector $\gamma^{-1}\frac{d\gamma}{ds}(s)$ which satisfies
\begin{equation*}
 g(\gamma^{-1}\frac{d\gamma}{ds}(s), B(s))=1
\end{equation*}
coincides with taking the Lorentz plane including $B(s)$.
Since determining a Lorentz plane is equivalent to determining a spacelike normal direction,
we can prove the following proposition.

\begin{Proposition}\label{betaneq0}
 For any minimal null scroll $\gamma(s) \cdot \exp( t\widetilde B(s) )$ which satisfies the minimality condition $g(A, B) = 2\beta$ with $\beta=1/2$,
 it follows
 \begin{equation*}
  \gamma^{-1} \frac{d\gamma}{ds}
  =
  -4\left(
       2g(B'', B'')B + B''
     \right)
  -
  2b\left(
      B' + \frac{b}4 B
     \right)
 \end{equation*}
 where $b= -2g(\gamma^{-1} \frac{d \gamma}{ds}, B')$.
 Conversely,
 for any curve $\widetilde B$ which takes values in the light cone in $\mathfrak{nil}_3$
 with $\beta=1/2$ and any real valued function $b$,
 define a $\mathfrak{nil}_3$-valued vector field $A$ as
 \begin{equation*}
  A= -4\left(
       2g(B'', B'')B + B''
     \right)
  -
  2b\left(
      B' + \frac{b}4 B
     \right).
 \end{equation*}
 Then the curve $\gamma$ which has the velocity $A$ is null
 and
 it follows that $g(A, B)=1$.
 Thus the null scroll over $\gamma$ with the ruling $\widetilde B$ is minimal.
\end{Proposition}
\begin{proof}
 Let $\gamma(s) \cdot \exp( t\widetilde B(s) )$ be a minimal null scroll
 which satisfies the minimality condition $g(A, B)=2\beta$ with $\beta=1/2$.
 Since
 \begin{align*}
  g(B, b B'+2B'') &= 2g(B, B'') = -\frac12 \neq0,\\
  g(2B'+bB, B) &= 0,\\
  g(2B'+bB, bB'+2B'') &= 2bg(B', B')+2bg(B, B'') = 0
 \end{align*}
 hold,
 we can see that the velocity $ \gamma^{-1}\frac{d\gamma}{ds} $ belongs to
 \begin{equation*}
 (\gamma^{-1}\frac{d\gamma}{ds} \times B)^{\perp}
 =
 (2B' + bB)^{\perp}
 =
 span\{B, bB' + 2B''\}
 \end{equation*}
 at each point.
 Denoting $\gamma^{-1}\frac{d\gamma}{ds}$ as $uB + v (bB' +2B'')$,
 simple computations of
 \begin{equation*}
 g \left( \gamma^{-1}\frac{d\gamma}{ds}, \gamma^{-1}\frac{d\gamma}{ds} \right) = 0
 \quad {\it and} \quad
 g \left( \gamma^{-1}\frac{d\gamma}{ds}, B \right) = 1
 \end{equation*}
 derive $u=-8g(B'', B'') - b^2/2$ and $v=-2$.
 Therefore we obtain
 \begin{equation*}
 \gamma^{-1} \frac{d\gamma}{ds}
  =
  -4\left(
       2g(B'', B'')B + B''
     \right)
  -2b \left( B'+\frac{b}4B \right).
 \end{equation*}
 Conversely,
 for an arbitrary curve $\widetilde B$
 which takes values in the light cone in $\mathfrak{nil}_3$ except for the origin
 and derives $\beta=1/2$,
 and an arbitrary function $b$,
 set a $\mathfrak{nil}_3$-valued vector field $A$ as
 \begin{equation*}
  A
  =
  -4\left(
       2g(B'', B'')B + B''
     \right)
  -2b \left( B'+\frac{b}4B \right).
 \end{equation*}
 By some straightforward computations,
 it can be seen that
 $A$ is the null vector field which belongs to $(2B' + bB)^{\perp} = span\{B, bB' + 2B''\}$
 and
 satisfies $g(A, B)=1$.
 Then by Theorem \ref{thm minimal condition},
 the null scroll over the null curve which has the velocity $A$ with the ruling $\widetilde B$ is minimal.
\end{proof}
\begin{Remark}\label{remreplaceb}
 The null vector field $A$ given in Proposition \ref{betaneq0} with $b=0$
 and $B$ define a null frame $(A, B, 2B')$,
 and then
 the curvatures are $k_1=4g(B'', B'')$ and $k_2=1/2$.
 Hence Example \ref{example_horizontalumbrella} shows
 if a curve $\widetilde B$ has $\beta=1/2$
 and $B''$ is null,
 $\widetilde B$ constructs a horizontal umbrella by Proposition \ref{betaneq0} with $b=0$ (see Example \ref{parabola ruling}).
 Moreover, although it is hard to compute the Abresch-Rosenberg differential in general,
 it can be obtained easily when $b=0$.
 Since we have \eqref{ARdiffcurvature} with null frame $(A, B, 2B')$,
 we can see $S=16{p_x}^2 k_1$. Therefore we have $Q=l64{p_x}^2 g(B'', B'')$. 
\end{Remark}
\begin{Example}[Circle ruling]
Set a ruling $\widetilde B=\frac12 e_1+ \frac12 \cos(s) e_2+\frac12 \sin(s) e_3$.
This describes a circle (left of Figure $3$).
Then we have $\beta=1/2$.
By Proposition \ref{betaneq0} with $b=0$,
the velocity of the base curve is given by
\begin{equation*}
 A=-e_1 + \cos(s) e_2 - \sin(s) e_3.
\end{equation*}
Therefore the base curve is computed as
\begin{equation*}
 \gamma(s)=\left( -s, \sin(s), -\frac12 s \sin(s) \right).
\end{equation*}
Hence the minimal null scroll $f(s, t)$ (right of Figure $3$) is given by
\begin{equation*}
 f(s, t)=
 \left( -s + \frac12 t, \sin(s) + \frac12 t \cos(s), -\frac12 s \sin(s) + \frac14 t \sin(s) - \frac14 t s \cos(s) \right).
\end{equation*}
\begin{figure}[t]
 \centering
 \includegraphics[width=5cm]{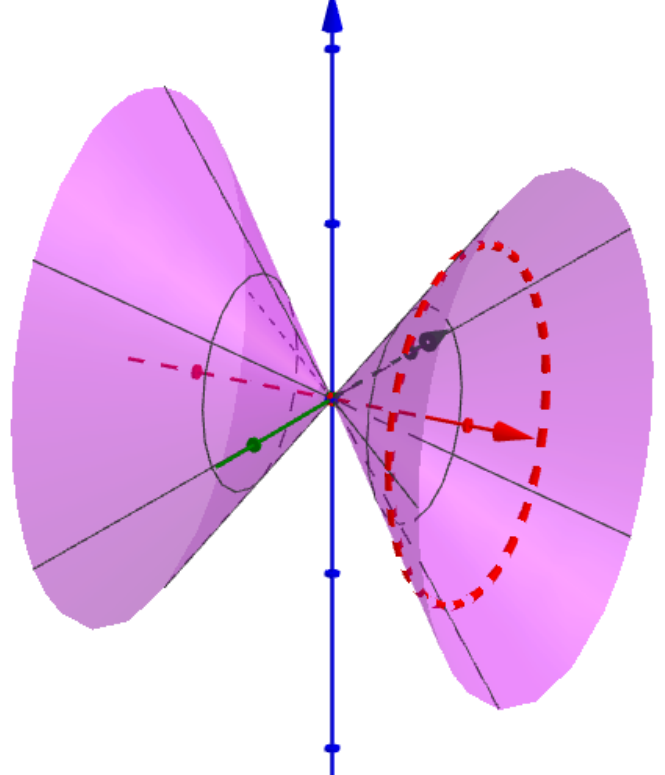}
 \includegraphics[width=5cm]{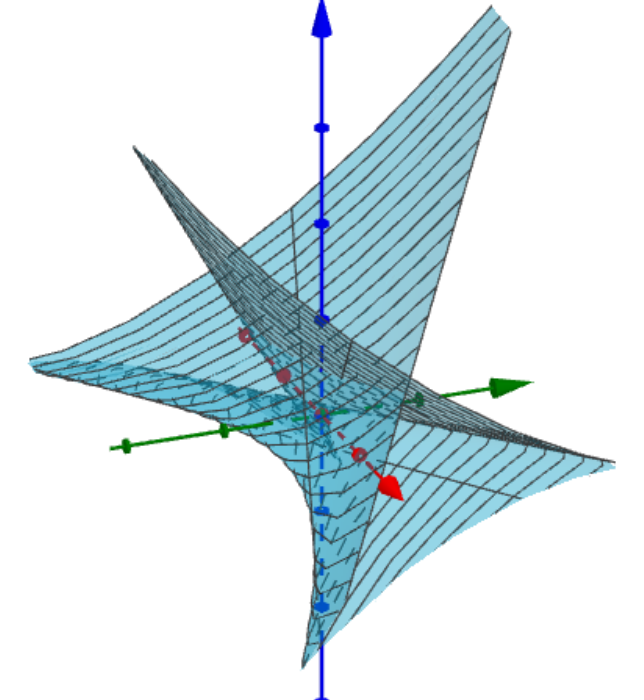}
 \caption{Example of minimal null scrolls constructed from ruling valued
 in a circle (right) and the image of its ruling (left, dash line).}
 \end{figure}
\end{Example}
\begin{Example}[Hyperbola ruling]
Set a ruling $\widetilde B=\frac12 \cosh(s) e_1+ \frac12 \sinh(s) e_2- \frac12 e_3$.
This describes a hyperbola (left of Figure $4$).
Then we have $\beta=\frac12$.
By Proposition \ref{betaneq0} with $b=0$,
the velocity of the base curve is given by
\begin{equation*}
 A=-\cosh(s) e_1 - \sinh(s) e_2 + e_3.
\end{equation*}
Therefore the base curve is computed as
\begin{equation*}
 \gamma(s)=\left( -\sinh(s), -\cosh(s) +1, \frac12 s +\frac12  \sinh(s) \right).
\end{equation*}
Hence the minimal null scroll $f(s, t)$ (right of Figure $4$) is given by
\begin{equation*}
 f(s, t)=
 \left( -\sinh(s) +\frac12 t \cosh(s), -\cosh(s) +1 + \frac12 t \sinh(s), \frac12 s +\frac12 \sinh(s) - \frac14 t -\frac14 t \cosh(s) \right).
\end{equation*}
\begin{figure}[h]
 \centering
 \includegraphics[width=5cm]{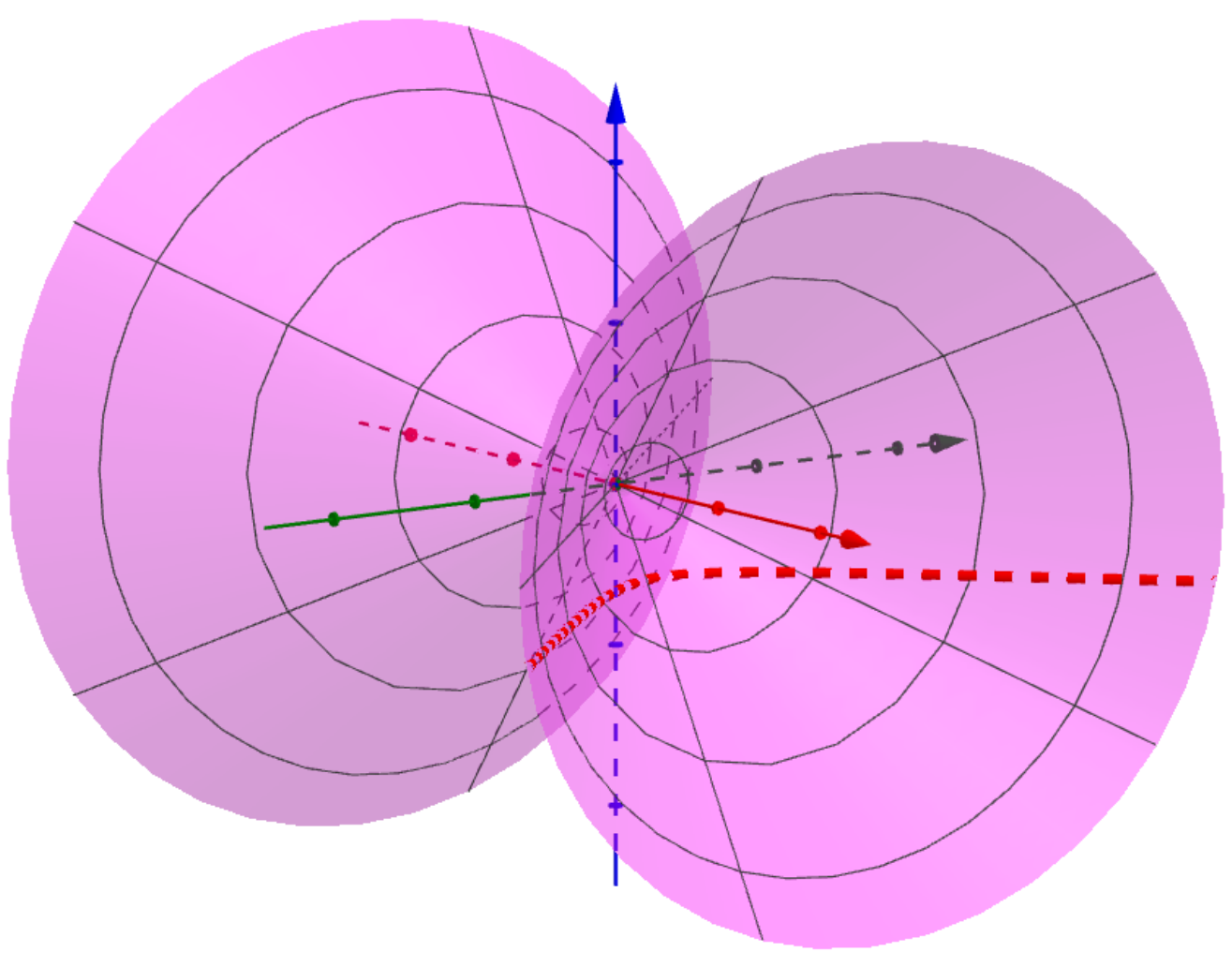}
 \includegraphics[width=5cm]{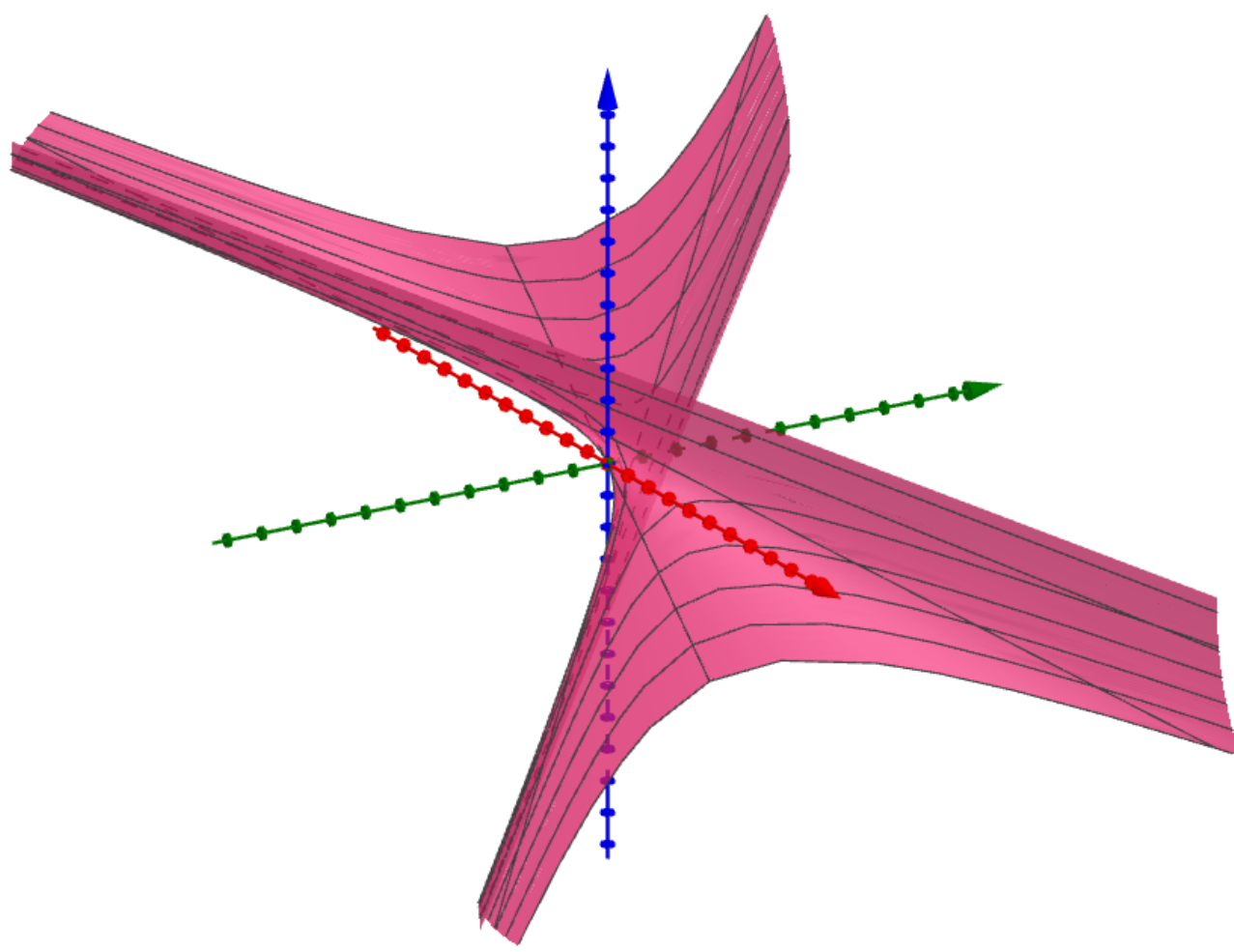}
 \caption{Example of minimal null scrolls constructed from ruling valued in a hyperbola (right)
 and the image of its ruling (left, dash line).}
 \end{figure}
\end{Example}
\begin{Example}[Parabola ruling]\label{parabola ruling}
Set a ruling $\widetilde B=\left( \frac18 s^2 + \frac12 \right) e_1+ \frac12 s e_2 +\left( \frac18 s^2 -\frac12 \right) e_3$.
Then we have $\beta=\frac12$.
This describes a parabola (left of Figure $5$).
By Proposition \ref{betaneq0} with a constant $b$,
the velocity of the base curve is given by
\begin{equation*}
 A=
 \left( -\frac{b^2}{16} s^2 -\frac{b}2 s -\frac{b^2}4 -1 \right) e_1
 -\left( \frac{b^2}4 s + b \right) e_2
 +\left( \frac{b^2}{16} s^2 + \frac{b}2 s - \frac{b^2}4 +1 \right) e_3.
\end{equation*}
Therefore the base curve is computed as
\begin{equation*}
 \gamma(s)=
 \left(
 -\frac{b^2}{48} s^3 - \frac{b}4 s^2 - \left( \frac{b^2}4  +1 \right)s,
 -\frac{b^2}8 s^2 -bs,
 -\frac{b^4}{3840} s^5 - \frac{b^3}{192} s^4 + \frac{b^4}{192} s^3 + \frac{b}4 s^2 + \left( 1-\frac{b^2}4 \right)s
 \right).
\end{equation*}
Hence the minimal null scroll $f(s, t)$ is given by
\begin{equation*}
 f(s, t)=
 \begin{pmatrix}
 -\frac{b^2}{48} s^3 - \frac{b}4 s^2 - \left( \frac{b^2}4  +1 \right)s +t \left( \frac18 s^2 + \frac12 \right),\\
 -\frac{b^2}8 s^2 -bs + \frac{st}2,\\
 -\frac{b^4}{3840} s^5 - \frac{b^3}{192} s^4 + \frac{b^4}{192} s^3 + \frac{b}4 s^2 + \left( 1-\frac{b^2}4 \right)s
 +t\left( \frac{b^2}{384} s^4 - \left( \frac18 + \frac{b^2}{32} \right)s^2 + \frac{b}4 s - \frac12 \right)
 \end{pmatrix}^\top.
\end{equation*}
If $b=0$, then the minimal null scroll $f$ lies on a plane in $\mathbb{R}^3$ (right of Figure $5$),
and this is also an example of horizontal umbrellas (see Example \ref{example_horizontalumbrella}).
\begin{figure}[h]
 \centering
 \includegraphics[width=5cm]{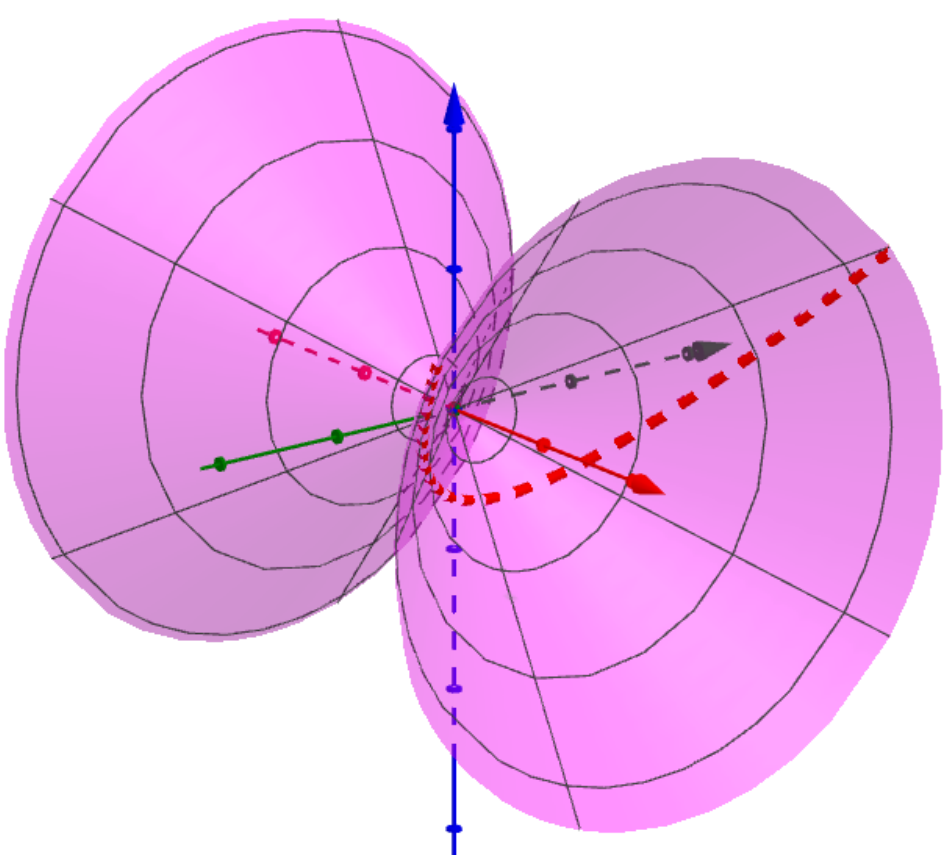}
 \includegraphics[width=5cm]{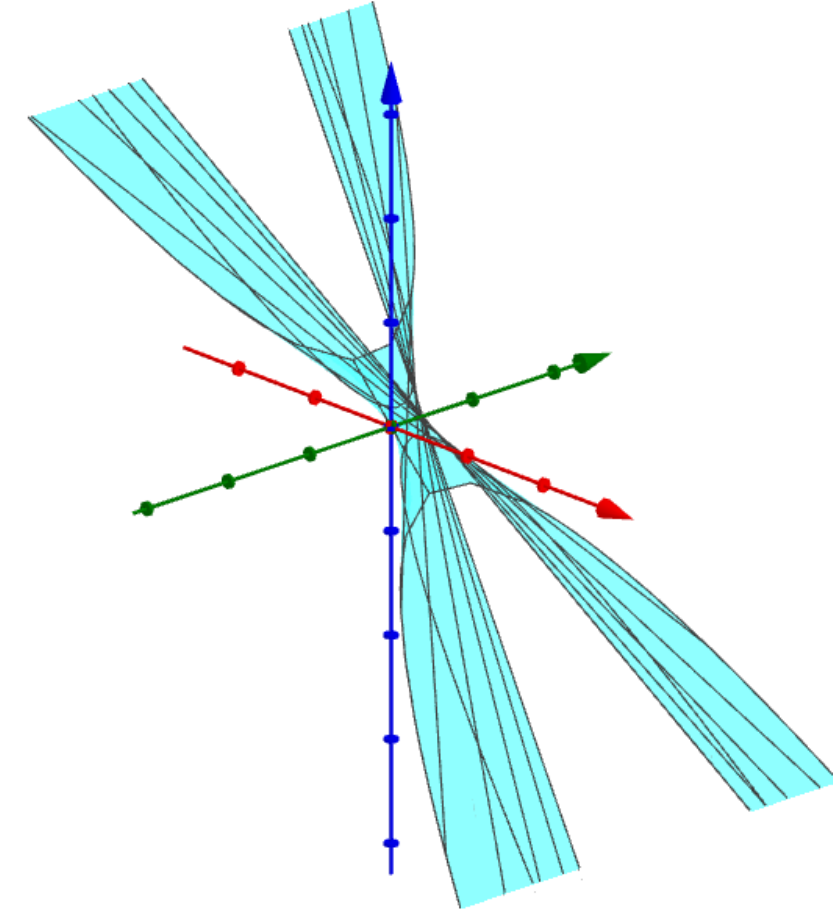}
 \caption{Example of minimal null scrolls constructed from ruling valued in a hyperbola (right)
 and the image of its ruling (left, dash line).}
 \end{figure}
\end{Example}


\subsection{Construction from the condition $g(A, \widetilde B) =0$}
Finally, we consider the case of $g(A, \widetilde B) =0$.
In this case,
obviously,
the rulings don't impose any conditions other than linear dependence on the velocity of the base curve.
Thus we have the proposition.
\begin{Proposition}\label{g(A, widetilde B) = 0}
 Let $\gamma(s) \cdot \exp( t \widetilde B(s) )$ be a minimal null scroll
 which satisfy the minimality condition $g(\gamma^{-1} \frac{d \gamma}{ds}, \widetilde B)=0$.
 Then it follows that $\beta \neq 0$ and $B^3 \neq 0$.
 Conversely,
 for a curve $\widetilde B$ in the light cone in $\mathfrak{nil}_3$
 and a real valued function $h$,
 define a null vector field $A=\sum_{i=1}^3 A^ie_i$ by $A=h\widetilde B$,
 and $\gamma$ as the curve which has the velocity $A$.
 If $\beta \neq 0$ and $B^3 \neq 0$ hold,
 then $\gamma(s) \cdot \exp( t \widetilde B(s) )$ is a minimal null scroll on $t\neq0$.
\end{Proposition}
\begin{proof}
 When $g(A, \widetilde B)=0$ holds,
 $A=h\widetilde B$ holds for some function $h$.
 Then we obtain
 \begin{equation*}
 A^1B^2-A^2B^1 =0.
 \end{equation*}
 Thus the non-degeneracy of null scrolls means
 \begin{equation*}
  g_{12} = t^2 \frac12 \beta (B^3)^2
 \end{equation*}
 vanishes nowhere,
 that is,
 $\beta \neq 0$, $B^3 \neq 0$, and $t \neq0$.
 Conversely,
 let $\gamma(s) \cdot \exp( t \widetilde B(s) )$ be a map
 which satisfies the condition
 $\gamma^{-1} \frac{d \gamma}{ds} = h \widetilde B$ for some function $h$.
 Obviously, $g(A, \widetilde B)=0$ holds.
 Assumptions $\beta\neq0$ and $B^3\neq0$ imply the non-degeneracy of the first fundamental form on $t\neq0$,
 and then the map $\gamma(s) \cdot \exp( t \widetilde B(s) )$ is a minimal null scroll on $t \neq0$.
\end{proof}
If the minimality condition $g(A, \widetilde B)=0$ holds, that is, $A=h\widetilde B$,
the minimal null scroll $\gamma(s) \cdot \exp( t \widetilde B(s)) = \gamma(s) \cdot \exp(\tilde{t} A(s))$ is of the form
\[
 \gamma(s) \cdot \exp(t \widetilde B(s))
 =(\gamma^1(s)+\tilde t{\gamma^1}'(s), \gamma^2(s)+\tilde t{\gamma^2}'(s), \gamma^3(s)+\tilde t{\gamma^3}'(s) ).
\]
Hence all the minimal null scrolls which satisfy the minimality condition $g(A, \widetilde B)=0$ are given by the following Example.
\begin{Example}[Tangent surfaces]
 Let $\gamma(s)=({\gamma^1}(s), {\gamma^2}(s), {\gamma^3}(s))$ be a curve in (semi-)Euclidean space $\R^3$.
 A surface $f(s, t)$ in $\mathbb{R}^3$ is said to be a {\it tangent surface} on $\gamma$ if
 $f$ is a ruled surface over $\gamma$, and its ruling is the velocity of $\gamma$,
 that is,
 the surface of the form
 \begin{align}
  f(s, t)
  &=\gamma(s)+t\gamma'(s)\notag\\
  &=(\gamma^1(s)+t{\gamma^1}'(s), \gamma^2(s)+t{\gamma^2}'(s), \gamma^3(s)+t{\gamma^3}'(s) ).\label{tangentsurface}
 \end{align}
 Let us regard $\gamma$ as a curve in $\Nil$.
 When $\gamma$ is null with respect to the metric $g$ on $\Nil$,
 $\gamma^{-1} \frac{d\gamma}{ds} \times (\gamma^{-1} \frac{d\gamma}{ds})' \neq 0$,
 and $g( \gamma^{-1} \frac {d\gamma}{ds}, e_3 ) \neq 0$,
 we can see the map defined in \eqref{tangentsurface} is a null scroll in $\Nil$.
 Moreover, it is minimal.
 In fact \eqref{tangentsurface} can be represented into
 \begin{equation*}
  \gamma(s) \cdot \exp ( t ( \gamma^{-1} \frac{d\gamma}{ds}(s) ) ).
 \end{equation*}
  Since $\gamma$ is null in $\Nil$ and Proposition \ref{g(A, widetilde B) = 0},
  tangent surface $f$ in $\R^3$ is a minimal null scroll in $\Nil$. 
\end{Example}
Propositions \ref{betaeq0}, \ref{betaneq0} and \ref{g(A, widetilde B) = 0} can be summarised
to obtain the construction theorem of minimal null scrolls with prescribed null rulings.
\begin{Theorem}\label{thmmain2}
 Let $\widetilde B= \sum_{i=1}^3 B^ie_i$ be a curve which takes values in the light cone in $\mathfrak{nil}_3$ except for the origin,
 and satisfies $\widetilde B \times \widetilde B' = -\beta \widetilde B$ with $\beta=0$ or $1/2$.
 Define a vector field $A$ if $\beta=0$ and $B^3 \neq 0$ as
 \begin{equation*}
  A = \alpha B,
 \end{equation*}
 and if $\beta =1/2$ as
 \begin{equation*}
  A
  = -\frac12b^2 B - 2b B' -4 \left( 2g(B'',B'')B + B'' \right).
 \end{equation*}
 Or else, if $\beta=1/2$ and $B^3 \neq 0$, define $A$ as 
 \begin{equation*}
 A= \alpha \widetilde B.
 \end{equation*}
 Here, $B= B^1e_1 + B^2e_2 - B^3e_3$,
 $b$ is an arbitrary function
 and $\alpha$ is a nowhere vanishing function.
 Moreover, let $\gamma$ be the curve in $\Nil$ which has the velocity $A$.
 Then the map $\gamma(s) \cdot \exp(t \widetilde B(s))$ locally defines a minimal null scroll.
\end{Theorem}
From Propositions \ref{betaeq0}, \ref{betaneq0} and \ref{g(A, widetilde B) = 0},
any minimal null scroll can be constructed locally by the Theorem \ref{thmmain2}.
We would like to note that
Theorem \ref{thmmain} implies
the class of minimal null scrolls with the minimality condition $g(A, \widetilde B)=0$
is included in the one with the minimality condition $g(A, B)=2\beta$ and $\beta \neq 0$.
However,
minimal null scrolls of the former class cannot be found in general through Theorem \ref{thmmain}
because no special features of the Abresch-Rosenberg differential and the support functions
are expected.


\bibliographystyle{plain}
\def\cprime{$'$}

\end{document}